\documentclass[journal]{IEEEtran}

\usepackage{mathtools}
\usepackage{amsmath}
\usepackage{cite}
\usepackage{amsfonts}
\usepackage{amssymb}
\usepackage{amsthm}
\usepackage{mathrsfs}
\usepackage{subfiles}
\usepackage{tikz}
\usepackage{tikzscale}
\usepackage{tkz-euclide} 
\usepackage{pgfplots}
\usetikzlibrary{calc,arrows,shapes.arrows,math,shapes.misc}
\usetikzlibrary{intersections}
\usetikzlibrary{decorations}
\usetikzlibrary{patterns.meta}
\usepackage{float}
\usepackage{xcolor}
\definecolor{ired}{rgb}{0.702, 0.176, 0.13}
\definecolor{iorange}{rgb}{0.93, 0.53, 0.18}
\definecolor{iyellow}{rgb}{0.984, 0.808, 0.125}
\definecolor{igreen}{rgb}{0.0, 0.42, 0.24}
\definecolor{iblue}{rgb}{0.05, 0.302, 0.549}
\definecolor{iltblue}{rgb}{0, 0.737, 0.894}
\definecolor{iltgray}{rgb}{0.796, 0.80, 0.796}
\definecolor{idkgray}{rgb}{0.349, 0.353, 0.349}
\definecolor{iviolet}{rgb}{0.45, 0.31, 0.59}
\definecolor{isand}{rgb}{0.76, 0.7, 0.5}
\definecolor{ibrown}{rgb}{0.51, 0.41, 0.33}
\usepackage{pgfplots}
\usepackage{bm,fixmath}
\usepackage{mathtools}
\usetikzlibrary{calc,angles,positioning,intersections,quotes,decorations.markings,decorations}
\usetikzlibrary{datavisualization.formats.functions}
\pgfplotsset{compat=1.17}

\DeclareMathOperator*{\argmin}{arg\,min}

\newcommand{\tobs}{t_{\text{obs}}}
\newcommand{\tapr}{t_{\text{apr}}}

\newtheorem{theorem}{Theorem}
\newtheorem{lemma}{Lemma}

\theoremstyle{definition}
\newtheorem{definition}{Definition}

\newcommand \DL{\mathrm{DL}}
\newcommand{\transpose}{\intercal}

\usepackage{tkz-euclide} 

\usepackage[colorlinks=true,pagebackref=false]{hyperref} 
\hypersetup{urlcolor=purple, citecolor=red, linkcolor=blue}

\usepackage[nameinlink]{cleveref}

\crefmultiformat{equation}{(#2#1#3)}%
{ and~(#2#1#3)}{, (#2#1#3)}{ and~(#2#1#3)}
\crefrangeformat{equation}{(#3#1#4)--(#5#2#6)}
\crefformat{equation}{(#2#1#3)}

\crefmultiformat{figure}{Fig.~#2#1#3}%
{ and~Fig.~#2#1#3}{, Fig.~#2#1#3}{ and~Fig.~#2#1#3}
\crefrangeformat{figure}{Figs.~#3#1#4 to~#5#2#6}
\crefformat{figure}{#2Fig.~#1#3}
\Crefformat{figure}{#2Fig.~#1#3}

\crefrangeformat{lemma}{Lemmas~#3#1#4--#5#2#6}
\crefrangeformat{theorem}{Theorems~#3#1#4--#5#2#6}


\begin{document}
	\title{Surveillance of a Faster Fixed-Course Target}
	
	\author{Isaac~E.~Weintraub,~\IEEEmembership{Senior Member,~IEEE,}
		Alexander~Von Moll,~\IEEEmembership{Member,~IEEE,}
		Eloy Garcia,~\IEEEmembership{Senior Member,~IEEE,}
		David W. Casbeer,~\IEEEmembership{Senior Member,~IEEE,}
		and~Meir Pachter,~\IEEEmembership{Life~Fellow,~IEEE}%
		\thanks{I. E. Weintraub is with the Control Science Center at Air Force Research Laboratory, Wright-Patterson AFB, OH, 45433 e-mail: isaac.weintraub.1@us.af.mil.}%
		\thanks{A. Von Moll, E. Garcia, and D. Casbeer are with the Control Science Center at Air Force Research Laboratory, Wright-Patterson AFB, OH, 45433}
		\thanks{M. Pachter is with the Department of Electrical and Computer Engineering at the Air Force Institute of Technology, Wright-Patterson AFB, OH, 45433}%
		\thanks{Manuscript received Month Day, Year; revised Month Day, Year.}}
	
	\markboth{Journal of IEEE Transactions on Aerospace and Electronic Systems}%
	{Shell \MakeLowercase{\textit{et al.}}: Bare Demo of IEEEtran.cls for IEEE Journals}
	
	\maketitle
	
	\begin{abstract}
		The maximum surveillance of a target which is holding course is considered, wherein an observer vehicle aims to maximize the time that a faster target remains within a fixed-range of the observer. This entails two coupled phases: an approach phase and observation phase. In the approach phase, the observer strives to make contact with the faster target, such that in the observation phase, the observer is able to maximize the time where the target remains within range. Using Pontryagin's Minimum Principle, the optimal control laws for the observer are found in closed-form. Example scenarios highlight various aspects of the engagement. 
	\end{abstract}
	
	
	\IEEEpeerreviewmaketitle
	
	\section{Introduction}
	\IEEEPARstart{S}{urveillance}-evasion problems are an important class of trajectory-planning problems wherein an evader chooses its trajectory to hinder the surveillance of an enemy observer. Throughout this paper the terms \textit{surveillance} and \textit{observation} are used interchangeably.
	In an early report by Koopman, tactics and scenarios surrounding search and and screening were described in detail~\cite{koopman1946search}.
	Koopman outlined naval and aerial strategies for searching out stationary and mobile targets using range-limited means such as visual detection, radar, and sonar. 
	In his report the location of targets were considered to be unknown by the observer, and the derived strategies leveraged probability to locate targets of interest through mobile search. 
	Later, Dobbie and Taylor posed and investigated surveillance-evasion~\cite{dobbie1966solution,dobbie1967survey,taylor1970application} making use of differential game theory~\cite{isaacs1965differential}. 
	Their work posed a turn-limited observer with greater speed than the target and defined the \emph{detection region}, \emph{surveillance region}, and \emph{escape region}. 
	
	More recent works have contributed to the surveillance-evasion differential game for a turn-limited observer~\cite{lewin1975surveillance, lewin1979conic, lewin1989isotropic, greenfeld1987differential}. 
	Where Dobbie and Taylor were concerned with the barriers in the game dictating guaranteed regions of observation or escape by a faster turn-limited observer, Lewin and Breakwell drew their attention to the game of degree~\cite{lewin1975surveillance} -- maximizing contact time assuming that the target was already in contact with the observer at the onset.
	Lewin and Olsder continued the work by changing the contact region from a circle to that of a two-dimensional cone~\cite{lewin1979conic}. 
	Lewin and Olsder also added more states to the original surveillance-evasion differential game by considering the isotropic rocket pursuit evasion game~\cite{bernhard1970linear}. 
	In the isotropic surveillance-evasion game, the target has bounded speed and can turn instantaneously and the observer has bounded acceleration and can direct it to any desired direction. 
	Greenfield constrained the target and the observer to have the same speed and turn radius.
	He investigated surveillance-evasion for the game of two-cars and presented the solution to the game of kind and the game of degree~\cite{greenfeld1987differential}. 
	Some 40 years after Koopman presented strategies for search under uncertainty; Gilles and Vladimirsky posed search-evasion as a differential game under uncertainty~\cite{gilles2020evasive}. 
	A recent paper by the authors considers a surveillance-evasion scenario with a model similar to that used in this paper, but the observer's control strategy is assumed to be pure pursuit~\cite{vonmoll2022pure}.
	
	Different from prior work, this paper models the target to be superior -- faster than the observer. 
	This presents the observer with limited opportunity to make contact with the target; and, optimal strategies for maximizing sustained contact are also investigated. 
	Prior work that considered a superior evader can be found in an early work by Breakwell~\cite{breakwell1975pursuit} wherein a single pursuer strives to capture a faster evader using point-capture. 
	Other works have considered multiple pursuers, in an effort to \emph{contain} the superior evader, include~\cite{hagedorn1976differential, szots2021revisiting, ramana2017pursuit-evasion,  garcia2021cooperative, jin2010pursuit-evasion}. 
	More specifically, references~\cite{hagedorn1976differential, szots2021revisiting} are focused on the so-called \textit{game of approach}, a zero-sum differential game in which the fast evader must pass between two pursuers.
	The aim of the pursuers is to minimize the approach distance, i.e., the minimum distance from either pursuer to the evader at any time along the trajectory.
	On the other hand, references~\cite{ramana2017pursuit-evasion, garcia2021cooperative, jin2010pursuit-evasion} are concerned with formations of pursuers which encircle the superior evader and determining whether there is a possibility for the evader to escape in any of the gaps between the pursuers.
	As pointed out in~\cite{chernousko1976problem}, point capture of a superior evader is not possible, even with multiple pursuers -- and thus these works endow a finite capture radius to the pursuers.
	This finite capture radius is akin to the observation disk modeled in this paper (hence the appearance of the Cartesian oval, both here and in~\cite{garcia2021cooperative}).
	
	Prior work in multi-phase pursuit-evasion scenarios have been considered, but not for surveillance-evasion~\cite{breakwell1979point, garcia2019strategies, garcia2018capture-the-flag, nath2022two-phase, shinar2009pursuit-evasion, turetsky2018pursuit-evasion,vonmoll2021turret, vonmoll2022turret-runner-penetrator}.
	In~\cite{breakwell1979point}, Breakwell and Hagedorn consider the point capture of two evaders in succession. In this multi-phase optimal control problem a faster pursuer aims to capture two equal-speed evaders in succession.
	Phase-I corresponds to the pursuit, and capture of, the first evader, followed by Phase-II wherein the second evader is captured.
	The cost/reward function is based upon the capture time of the second evader, which implicitly depends on the agents' trajectories in Phase-I.
	A similar setup is found in~\cite{garcia2019strategies}, but rather than maximizing capture time, the evaders seek to minimize their terminal $y$-coordinate.
	
	In capture-the-flag games, the agents' strategies change between two phases. In Phase-I, an attacker strives to capture an enemy's flag while an opposing defender tries to intercept the attacker before the attacker can reach the flag. 
	In Phase-II (conditioned by the outcome of Phase-I) the attacker then attempts to reach a safe zone while the defender strives to reach the attacker before the attacker can reach the safe zone~\cite{garcia2018capture-the-flag, huang2015automation, huang2011differential}. 
	The control strategies for both sides vary based upon which phase of the game is currently active.
	
	In the work of Nath, a pursuit-evasion scenario between two non-holonomic agents is considered ~\cite{nath2022two-phase}.
	The scenario itself can be considered to be a single phase, but the proposed evader strategy is comprised of two phases: one for large separation distance, and one for when the pursuer is nearby. In the reference ~\cite{garcia2021differential}, a multi-player differential game is described where the overall game is played in two phases or stages, the attack stage and the retreat stage.
	
	Additionally, works by Shinar and Turetsky have focused on pursuit-evasion for hybrid and switched systems in which the cost/reward is based on the zero-effort miss distance~\cite{shinar2009pursuit-evasion, turetsky2018pursuit-evasion}.
	In~\cite{shinar2009pursuit-evasion}, the pursuer has two sets of dynamics, which are both known to the evader, and can switch once during the game.
	However, the evader does not know when the pursuer's dynamics switch.
	This setup is extended in ~\cite{turetsky2018pursuit-evasion} wherein the pursuer dynamics may switch many times; both the full information (where evader knows when the switches occur) and asymmetric information cases are considered.
	
	Recently, some turret defense scenarios comprised of multiple phases have been analyzed~\cite{vonmoll2021turret, vonmoll2022turret-runner-penetrator}.
	In both scenarios, the transition between phases occurs when the turret aligns its look angle with a mobile agent.
	The solution methodology involves obtaining the Value function for Phase-II and using it as the terminal cost/reward for Phase-I.
	For example, in~\cite{vonmoll2021turret}, the mobile agent is an attacker which must decide between engaging the turret or retreating to a prespecified safe zone.
	The attacker's instantaneous cost (for either choice) is a piecewise discontinuous function which depends on whether the turret is aligned or not.
	Meanwhile, in~\cite{vonmoll2022turret-runner-penetrator}, two attackers cooperate against the turret.
	In Phase-I, one of the attackers draws the turret away from the second attacker in an effort to better the position of the latter in Phase-II (once the first attacker has been neutralized).
	
	Other relevant works have considered pursuit to consist of multiple phases to maximize an objective concerning and entire pursuit-evasion scenario. 
	In \cite{turetsky2019minimum}, a pursuing missile delays a target assignment to maximize the expectation of capture of two possible evaders. 
	In Phase-I a pursuer moves toward a virtual target then makes a decision for which evader to pursue. 
	In Phase-II, the pursuer engages a specific targeted evader. 
	Later, this target assignment was scaled to consider multiple pursuers and many evaders, \cite{weiss2022minimum}.
	
	This paper extends prior work concerning only the observation phase; solved in~\cite{weintraub2020maximum, weintraub2021maximum} to include the approach phase and answer how the approach is related to the observation phase.
	The relationship between the phases is that the Value of Phase-II is treated as the terminal cost for Phase-I, giving rise to the optimal control for the whole scenario via the principle of dynamic programming.
	The observer's circular range has realistic implications including, but not limited to: 1) visual contact range, 2) sensor range, and or 3) communication range.
	Moreover, the target is assumed to have a fixed course, coinciding with two realistic possibilities: 1) the target is unaware of the observer (and thus makes no maneuver to avoid the latter), or 2) the target's is far less maneuverable than the observer and therefore its trajectory can be approximated as a straight line.
	This assumption appears also in several examples within the seminal work on missile control by Shneydor~\cite{shneydor1998missile}, as well as in examples in works by Barton and Eliezer concerning a pursuer that implements pure pursuit~\cite{barton2000pursuit}.

	In order to solve the two-phase optimal control problem for maximizing surveillance of a faster fixed-course target, optimal control theory is leveraged, specifically Pontryagin's Minimum Principle,~\cite{kirk1970optimal}. In \Cref{sec:OptimalControl} the optimal control problem is defined.
	In \Cref{sec:ObservationPhase}, the optimal strategy for Phase-II is solved; then, in \Cref{sec:ApproachPhase} the optimal strategy for the observer in Phase-I is solved. The unified optimal strategy for both phases is presented in \Cref{sec:CompleteSolution}.
	Three scenarios are presented in \Cref{sec:Examples}, highlighting the optimal strategy for the observer for various initial conditions.
	Lastly, in \Cref{sec:Conclusions}, concluding remarks and future extensions are identified.
	\section{Problem Formulation} \label{sec:OptimalControl}
	Consider the optimal observation of a faster, non-maneuvering target by a slower observer. The observer seeks to maximize the observation time of the target by controlling its heading. The speed of the observer ($O$) and the target ($T$) are $v_O$ and $v_T$ respectively. Also, define the speed ratio parameter: $\alpha \triangleq v_O/v_T$. Because the observer is slower than the target, $0<\alpha<1$. Further, and without loss of generality, consider a Cartesian coordinate frame whose y-axis is aligned with the velocity vector of the non-maneuvering constant-speed target. The state of the observer-target scenario is
	\begin{equation}
		\mathbf{x}(t) = [x_O(t),y_O(t),y_T(t)]^\intercal \in \mathbb{R}^3,
		\label{eq:state}
	\end{equation}
	see \Cref{fig:overallGeometry}.
	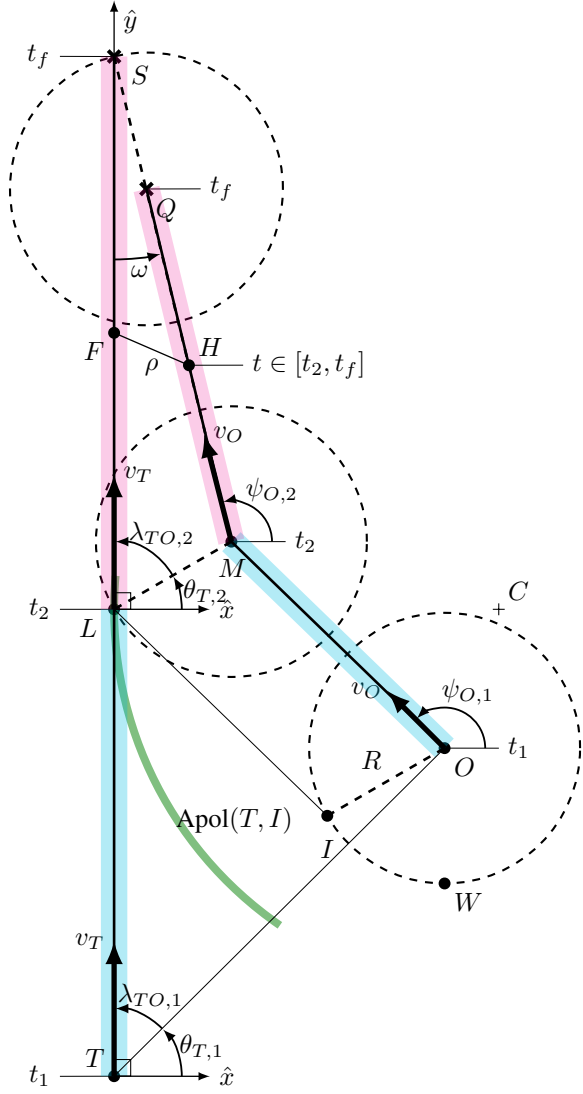
\begin{figure}[htbp]
		\centering
		\definecolor{ired}{rgb}{0.702, 0.176, 0.13}
\definecolor{iorange}{rgb}{0.93, 0.53, 0.18}
\definecolor{iyellow}{rgb}{0.984, 0.808, 0.125}
\definecolor{igreen}{rgb}{0.0, 0.42, 0.24}
\definecolor{iblue}{rgb}{0.05, 0.302, 0.549}
\definecolor{iltblue}{rgb}{0, 0.737, 0.894}
\definecolor{iltgray}{rgb}{0.796, 0.80, 0.796}
\definecolor{idkgray}{rgb}{0.349, 0.353, 0.349}
\definecolor{iviolet}{rgb}{0.45, 0.31, 0.59}
\definecolor{isand}{rgb}{0.76, 0.7, 0.5}
\definecolor{ibrown}{rgb}{0.51, 0.41, 0.33}

\begin{tikzpicture}[scale=1.8, vect/.style={->,shorten >=0.5pt,>=latex, line width = 0.2}]


\tkzDefPoint({0.7071+1.7330},{0.7071-1.7330}){O} 	
\tkzDefPoint({0.7071+1.7330},{0.7071-1.7330-1.00}){W}
\tkzDefPoint(0.866025,0.5){M} 		
\tkzDefPoint(0.2389,3.1059){Q} 		
\tkzDefPoint(3.4401,-1.0259){R} 	
\tkzDefPoint(0,-3.45){T} 			
\tkzDefPoint(0,0){L} 				
\tkzDefPoint(0,4.0847){S} 			
\tkzDefPoint({2.4401+0.4},{0}){C} 	
\tkzDefPoint({0.7071+1.7330-0.866},{0.7071-1.7330-0.50}){I}

\tkzDefPoint(2.0,0.7071){xhat} 		
\tkzDefPoint(0,4.5){yhat}
\tkzDefPoint(0.7071,0){xhatT} 		
\tkzDefPoint(0.4,0){timeR} 			
\tkzDefPoint(-0.4,0){timeL} 			

\tkzDefPoint(0,{-3.45+1.0}){v_{T,1}}
\tkzDefPoint(0.6642,1.2944){v_{O,2}}
\tkzDefPoint({2.4401-0.4243},{-1.0259+0.4243}){v_{O,1}}
\tkzDefPoint(0,1){v_{T,2}}

\draw(v_{O,1}) node [right, scale = 1, shift={(-0.6,0)}]{$v_O$};
\draw(v_{O,2}) node [right, scale = 1, shift={(0,0)}]{$v_O$};
\draw(v_{T,2}) node [right, scale = 1, shift={(0,0)}]{$v_T$};
\draw(v_{T,1}) node [left, scale = 1, shift={(0,0)}]{$v_T$};

\tkzDrawPoints[shape = circle, scale = 1, line width = 1.5](M)
\tkzDrawPoints[shape = circle, scale = 1, line width = 1.5](L)
\tkzDrawPoints[shape = cross out, scale = 1, line width = 1.5](Q)
\tkzDrawPoints[shape = cross out, scale = 1, line width = 1.5](S)
\tkzDrawPoints[shape = circle, scale = 1, line width = 1.5](T)
\tkzDrawPoints[shape = circle, scale = 1, line width = 1.5](O)
\tkzDrawPoints[shape = circle, scale = 1, line width = 1.5](W)

\tkzDrawPoints[shape = circle, scale = 1, line width = 1.5](I) 	%
\tkzDrawPoints[shape = cross, scale = 1](C) 					

\tkzDrawArc[line width=3,color=green!50!black, opacity=0.5, delta=5](C,L)(I)
\tkzDrawCircle[dashed,line width=0.8](M,L)
\tkzDrawCircle[dashed,line width=0.8](Q,S)
\tkzDrawCircle[dashed,line width=0.8](O,R)

\tkzDefPointWith[colinear=at O](L,timeR)
\tkzGetPoint{timeR1}
\tkzDefPointWith[colinear=at M](L,timeR)
\tkzGetPoint{timeR2}
\tkzDefPointWith[colinear=at Q](L,timeR)
\tkzGetPoint{timeR3}

\tkzDefPointWith[colinear=at S](L,timeL)
\tkzGetPoint{timeL3}
\tkzDefPointWith[colinear=at T](L,timeL)
\tkzGetPoint{timeL2}
\tkzDefPointWith[colinear=at L](L,timeL)
\tkzGetPoint{timeL1}
\tkzDefPointWith[colinear=at T](L,xhatT)
\tkzGetPoint{xhat0}

\tkzDrawSegments[-,line width=0.2](L,timeL)
\tkzDrawSegments[-,line width=0.2](S,timeL3)
\tkzDrawSegments[-,line width=0.2](T,timeL2)
\tkzDrawSegments[-,line width=0.2](Q,timeR3)
\tkzDrawSegments[-,line width=0.2](M,timeR2)
\tkzDrawSegments[-,line width=0.2](O,timeR1)

\tkzDrawSegments[color=iltblue,line width=10,opacity=0.3](T,L)
\tkzDrawSegments[color=iltblue,line width=10,opacity=0.3](O,M)
\tkzDrawSegments[line width=1](T,L)
\tkzDrawSegments[line width=1](O,M)
\tkzDrawSegments[line width=10,color=magenta,opacity=0.2](M,Q)
\tkzDrawSegments[line width=10,color=magenta,opacity=0.2](L,S)
\tkzDrawSegments[line width=1](M,Q)
\tkzDrawSegments[line width=1](L,S)
\tkzDrawSegments[dashed,line width=1](M,S)
\tkzDrawSegments[dashed,line width=1](M,L)
\tkzDrawSegments[dashed,line width=1](I,O)
\tkzDrawSegments[line width=0.3](T,O)
\tkzDrawSegments[line width=0.3](I,L)

\tkzDrawSegments[->,>=latex,line width=2.0](O,v_{O,1})
\tkzDrawSegments[->,>=latex,line width=2.0](M,v_{O,2})
\tkzDrawSegments[->,>=latex,line width=2.0](L,v_{T,2})
\tkzDrawSegments[->,>=latex,line width=2.0](T,v_{T,1})
\tkzDrawSegments[->,>=latex,line width=0.5](L,xhatT)
\tkzDrawSegments[->,>=latex,line width=0.5](T,xhat0)
\tkzDrawSegments[->,>=latex,line width=0.5](S,yhat)

\tkzLabelPoints[shift = {(-0.3,-0.1)}](M)
\tkzLabelPoints[](Q)
\tkzLabelPoints[below left, shift = {(-0.1,0)}](L)
\tkzLabelPoints[shift = {(0.1,0)}](S)
\tkzLabelPoints[above left](T)
\tkzLabelPoints[](O)
\tkzLabelPoints[below, shift = {(0.0,-0.2)}](I)
\tkzLabelPoints[](W)
\tkzLabelPoints[above right](C)

\draw($(I)!0.44!(O)$) node [above, scale = 1, shift={(-0.1,0.1)}]{$R$};
\draw(timeL1) node [left, scale = 1, shift={(0,0)}]{$t_2$}; 		
\draw(timeL2) node [left, scale = 1, shift={(0,0)}]{$t_1$}; 		
\draw(timeL3) node [left, scale = 1, shift={(0,0)}]{$t_f$}; 		
\draw(timeR1) node [right, scale = 1, shift={(0,0)}]{$t_1$}; 		
\draw(timeR2) node [right, scale = 1, shift={(0,0)}]{$t_2$};  		
\draw(timeR3) node [right, scale = 1, shift={(0,0)}]{$t_f$};		
\draw(T) node [right, scale = 1, shift={(0.7,3.4)}]{Apol$(T,I)$}; 	

\draw(xhatT) node [right]{$\hat{x}$};
\draw(xhat0) node [right]{$\hat{x}$};
\draw(yhat) node [below right]{$\hat{y}$};

\tkzMarkAngle[size = 0.3, color = black,vect, mark=0, line width=0.7](timeR2,M,S)
\tkzLabelAngle[pos = 0.45, scale=1, shift={(0.1,0)} ](xhat,M,v_{O,2}){$\psi_{O,2}$}

\tkzMarkRightAngle[scale = 0.7](xhatT,L,v_{T,2})
\tkzMarkAngle[size = 0.5, color = black,vect, mark=0, line width=0.7](xhatT,L,M)
\tkzMarkAngle[size = 0.5, color = black,vect, mark=0, line width=0.7](M,L,v_{T,2})

\tkzLabelAngle[pos = 0.7, scale=1, shift={(0.0,-0.1)} ](xhatT,L,M){$\theta_{T,2}$}
\tkzLabelAngle[pos = 0.7, scale=1, shift={(0.0,-0.1)} ](M,L,v_{T,2}){$\lambda_{TO,2}$}

\tkzMarkAngle[size = 0.3, color = black,vect, mark=0, line width=0.7](timeR1,O,v_{O,1})
\tkzLabelAngle[pos = 0.45, scale=1, shift={(0,0)} ](timeR1,O,v_{O,1}){$\psi_{O,1}$}

\tkzMarkRightAngle[scale = 0.7](xhat0,T,v_{T,1})
\tkzMarkAngle[size = 0.5, color = black,vect, mark=0, line width=0.7](xhat0,T,O)
\tkzMarkAngle[size = 0.5, color = black,vect, mark=0, line width=0.7](O,T,v_{T,1})
\tkzLabelAngle[pos = 0.7, scale=1, shift={(0.0,-0.1)} ](xhat0,T,O){$\theta_{T,1}$}
\tkzLabelAngle[pos = 0.7, scale=1, shift={(0.0,-0.1)} ](O,T,v_{T,1}){$\lambda_{TO,1}$}

\tkzMarkAngle[size = 1.5, color = black,vect, mark=0, line width=0.7](L,S,M)
\tkzLabelAngle[pos = 1.6, scale=1](L,S,M){$\omega$}

\tkzDefMidPoint(L,S) 
\tkzGetPoint{F}
\tkzDrawPoints[shape = circle, scale = 1, line width = 1.5](F)

\tkzDefMidPoint(M,Q) 
\tkzGetPoint{H}
\tkzDrawPoints[shape = circle, scale = 1, line width = 1.5](H)
\tkzDrawSegment[line width=0.5](H,F)
\draw($(H)!0.5!(F)$) node [below, scale = 1, shift={(0,0)}]{$\rho$};
\tkzDefPointWith[colinear=at H](L,timeR)
\tkzGetPoint{timeR4}
\tkzDefPointWith[colinear=at F](L,timeL)
\tkzGetPoint{timeL4}

\draw(timeR4) node [right, scale = 1, shift={(0,0)}]{$t\in[t_2,t_f]$}; 
\tkzLabelPoints[below left](F)
\tkzLabelPoints[above right](H)

\tkzDrawSegments[-,line width=0.2](H,timeR4)

\end{tikzpicture}
		\caption{In the optimal observation maneuver entails two phases: 1) approach and 2) observation. In the approach phase (blue) the observer aims to bring the target within range. In the observation phase (pink) the observer aims to keep the fast target in range for as long as possible.}
		\label{fig:overallGeometry}
	\end{figure} 
	The target is said to be \emph{observed} when it is within a range, $R\; (>0)$, of the observer. The observation set, $\mathscr{O}$, is modeled as a disk of radius, $R$,
	\begin{equation}
		\mathscr{O} = \lbrace \mathbf{x}  (t) \ | \ x_O^2(t) +  (y_O(t) - y_T(t))^2 - R^2 \leq 0, t \in [t_2,t_f] \rbrace.
		\label{eq:observation}
	\end{equation}
	The positions of $O$ and $T$ in the reference frame are specified by the Cartesian coordinates $(x_O(t),y_O(t))$ and $(0,y_T(t))$, respectively. The observer's control variable is its instantaneous heading angle $u(t) = \psi_O(t)$. The explicit time dependence of the states, costates, and control will henceforth be suppressed. The dynamics for this scenario are
	\begin{equation}
		\dot{x}_O = \alpha \cos \psi_O, \quad \dot{y}_O = \alpha \sin \psi_O, \quad \dot{y}_T = 1.
		\label{eq:dynamics}
	\end{equation}
	Since the target is faster than the observer, escape of the target from the observer is unavoidable. Escape occurs when $T$ is no longer within the range, $R$, of $O$; this determines the final time, $t_f$. Thus, the terminal manifold is
	\begin{equation}
		\mathscr{C} = \lbrace \mathbf{x} \;|\; R^2 - x_O^2 - (y_O-y_T)^2 < 0  ,  t \geq t_2\rbrace.
		\label{eq:Escape}
	\end{equation}
	The objective of the observer is to maximize the time for which the target remains inside its observation disk,
	\begin{equation}
		\textstyle \psi_{O}^*(t) = \argmin_{\psi_{O}} J = \int_{t_2}^{t_f} - 1 \mathrm{d}t = t_2-t_f.
		\label{eq:ObjectiveCost}
	\end{equation}
	The cost functional in \cref{eq:ObjectiveCost} states that the observer aims to choose headings that maximize the observation time in the second phase.
	The observation time in Phase-II inherently depends on the observer's trajectory in Phase-I.
	In order to model the observer's objectives for Phase-I and Phase-II, specify the approach time as $\tapr \triangleq t_2-t_1$ and the observation time $\tobs \triangleq t_f-t_2$. These time intervals pertain to Phase I and Phase-II, respectively. The two phases occur in succession, and therefore: $t_f = \tapr + \tobs$, where $t_1$, being the instant of initiation of the engagement, is set to $t_1=0$. The analysis starts with Phase-II or ``end game''.
	
	
	\section{Phase - II: Observation Phase} \label{sec:ObservationPhase}
	First, consider the observation phase wherein the observer aims to maximize the time that the target remains inside the observation disk. The Hamiltonian for Phase-II is 
	\begin{equation}
		\mathscr{H}_{\text{II}} = p_{x_O} \alpha \cos \psi_{O,2} + p_{y_O} \alpha \sin \psi_{O,2}  + p_{y_T}
		\label{eq:Hamiltonian1}
	\end{equation}
	and the costates are
	\begin{equation}
		\mathbf{p} = [p_{x_O} \; p_{y_O} \; p_{y_T}]^\intercal .
		\label{eq:costates}
	\end{equation}
	The Pontraygin Minimum Principle (PMP) yields necessary conditions for optimality.
	\begin{align}
		\dot{\mathbf{x}}^*(t) &= \tfrac{\partial \mathscr{H}(\mathbf{x}^*(t),\mathbf{p}(t),\psi_{O,2}^*(t),t)}{\partial \mathbf{p}}, \label{eq:NOCx} \\ 
		\dot{\mathbf{p}}(t) &= -\tfrac{\partial \mathscr{H}(\mathbf{x}^*(t),\mathbf{p}(t),\psi_{O,2}^*(t),t)}{\partial \mathbf{x}},\label{eq:NOCy} \\ 
		\mathbf{0} &= \tfrac{\partial \mathscr{H}(\mathbf{x}^*(t),\mathbf{p}(t),\psi_{O,2}^*(t),t)}{\partial \psi_{O,2}}.
		\label{eq:NOCu}
	\end{align}
	and $\mathscr{H}_{\text{II}}(t_f)=0$. The superscript, $*$ represents optimality. Evaluating the stationarity condition specified in \cref{eq:NOCu},
	\begin{equation}
		\begin{aligned}
			0  &= p_{y_O} \cos \psi_{O,2}^* - p_{x_O} \sin \psi_{O,2}^*.
		\end{aligned}
		\label{eq:NOCu1}
	\end{equation}
	Evaluating the necessary conditions in \cref{eq:NOCy}, the costates are found to be constant, as expected:
	\begin{equation}
		\dot{p}_{x_O}(t) = 0,\quad
		\dot{p}_{y_O}(t) = 0,\quad
		\dot{p}_{x_T}(t) = 0 \label{eq:constLambda}.
	\end{equation}
	Solving for the optimal observer's heading $\psi_{O,2}^*$ using \cref{eq:NOCu1}, the following is obtained:
	\begin{equation}
		\begin{aligned}
			\sin \psi_{O,2}^* = \tfrac{p_{y_O}}{\sqrt{p_{x_O}^{2}+p_{y_O}^{2}}}, \quad \cos \psi_{O,2}^* = \tfrac{p_{x_O}}{\sqrt{p_{x_O}^{2}+p_{y_O}^{2}}}
		\end{aligned}
		\label{eq:optimalControl1}
	\end{equation}
	From \cref{eq:optimalControl1,eq:constLambda} the costates are constant therefore the optimal heading of the observer is constant. 
	
	Next, consider the transversality conditions which are used to formulate the relationship between the states and costates at final time, $t_f$.
	\begin{equation}
		\tfrac{\partial h}{\partial \mathbf{x}}(\mathbf{x}^*(t_f),t_f) - \mathbf{p}(t_f) = \sigma \tfrac{\partial m}{\partial \mathbf{x}}(\mathbf{x}^*(t_f),t_f);
		\label{eq:transversality}
	\end{equation} 
	where $h(\cdot)$ is the terminal cost of the object functional, $\sigma$ is a slack variable, and $m$ is the terminal manifold, \cref{eq:Escape};
	\begin{equation}
		m(\mathbf{x}(t_f),t_f) = x_O^{*2}(t_f) + (y_O^*(t_f) - y_T^*(t_f))^2 - R^2 = 0.
	\end{equation}
	Evaluation of the transversality conditions in \cref{eq:transversality}, 
	\begin{equation}
		-\mathbf{p}^*(t_f) = \sigma \begin{bmatrix}
			\tfrac{\partial m}{\partial x_O}&\tfrac{\partial m}{\partial y_O}&\tfrac{\partial m}{\partial y_T}	
		\end{bmatrix}^\transpose_{t=t_f}.
	\end{equation}
	Therefore,
	\begin{equation}
		\begin{aligned}
			p_{x_O}^*(t_f) &= -\sigma \tfrac{\partial m}{\partial x_O}\big\vert_{t=t_f} = -2\sigma x_O(t_f),\\
			p_{y_O}^*(t_f) &= -\sigma \tfrac{\partial m}{\partial y_O}\big\vert_{t=t_f} = 2\sigma (y_T(t_f)-y_O(t_f)),\\
			p_{y_T}^*(t_f) &= -\sigma \tfrac{\partial m}{\partial y_T}\big\vert_{t=t_f} = 2\sigma (y_O(t_f)-y_T(t_f))
		\end{aligned}.
		\label{eq:costatePhase2}
	\end{equation}
	Substitution of $p_{x_O}$ and $p_{y_O}$ from \cref{eq:costatePhase2} into the optimal heading obtained in \cref{eq:optimalControl1}, the optimal heading for the second phase is
	\begin{equation}
		\sin \psi_{O,2}^*(t_f) = \tfrac{\pm (y_T(t_f)-y_O(t_f))}{R}.
		\label{eq:optimalHeading2}
	\end{equation}
	When the sign of \cref{eq:optimalHeading2} is $+$, the observer is heading toward the target and when the sign is $-$, the observer is headed away from the target, which is not a case of interest. The observer needs to have a positive component of velocity along the target's direction for it to be a viable optimal path. Because of \cref{eq:optimalHeading2}, it is observed that $\overrightarrow{MQ}$ and $\overrightarrow{QS}$ are collinear. Therefore, one can make use of $\triangle LMS$ for analysis in Phase-II. More details concerning Phase-II are in an earlier work~\cite{weintraub2020maximum}, and for the 3-D case, ~\cite{weintraub2021maximum}.
	
	\begin{lemma}
		The optimal heading of the observer that maximizes observation time for Phase-II is $\psi_{O,2}^* = \cos^{-1} \left( \tfrac{(\alpha^2-1)\sin \lambda_{TO,2}}{\alpha^2 + 2 \alpha \cos \lambda_{TO,2} + 1} \right)$, where $\alpha \in (0,1)$ is the speed ratio between the observer and the target and $\lambda_{TO,2} \in [-\pi, \pi]$ is the relative bearing from the target to the observer.
		\label{lma:OptimalHeadingII}
	\end{lemma}
	\begin{proof}
		Using the law of cosines for the triangle $\triangle MSL$ from \cref{fig:overallGeometry}:
		\begin{equation}
			\overline{MS}^2 = \overline{LS}^2 + R^2 - 2R \overline{LS}\cos \lambda_{TO,2}.
			\label{eq:MS2}
		\end{equation}
		From the speed ratio:
		\begin{equation}
			\overline{MS} = \alpha \overline{LS} + R.
			\label{eq:MS}
		\end{equation}
		Substitution of \cref{eq:MS} into \cref{eq:MS2}, and solving for $\overline{LS}$, the following is obtained:
		\begin{equation}
			\overline{LS} = \tfrac{2 R (\alpha + \cos \lambda_{TO,2})}{1 - \alpha^2}
			\label{eq:LS1}
		\end{equation}
		Since the cosine of an angle is the adjacent distance over the hypotenuse, the following is obtained:
		\begin{equation}
			\cos(\pi - \psi_{O,2}) = \tfrac{R \sin \lambda_{TO,2}}{\alpha \overline{LS} + R}
			\label{eq:LS2}
		\end{equation}
		Inserting \cref{eq:LS1} into \cref{eq:LS2}, 
		\begin{equation}
			- \cos \psi_{O,2} = \tfrac{R \sin \lambda_{TO,2}}{\alpha \tfrac{2 R (\alpha + \cos \lambda_{TO,2})}{1 - \alpha^2} + R}.
			\label{eq:psiO21}
		\end{equation}
		Through algebraic manipulation of \cref{eq:psiO21} one finally obtains the optimal observer heading for the second phase:
		\begin{equation}
			\psi_{O,2}^* = \cos^{-1} \begin{pmatrix} \tfrac{(\alpha^2-1)\sin\lambda_{TO,2}}{\alpha^2 + 2 \alpha \cos \lambda_{TO,2} + 1}
			\end{pmatrix}
			\label{eq:solution_phase2}
		\end{equation}
	\end{proof}
	
	\begin{lemma}
		Once the target is within the range of the observer (under optimal play), the target remains inside the observation disk until the state, $\mathbf{x}(t)$, reaches the terminal manifold, $\mathscr{C}$ at time, $t_f$ -- observation is invariant.
		\label{lma:Invariance}
	\end{lemma}
	
	\begin{proof}
		Using the law of cosines for $\triangle MSL$, the following relationship is obtained:
		\begin{equation} 
			\label{eq:invar1}
			R^2 = \overline{LS}^2 + (\alpha \overline{LS} + R)^2 - 2 \overline{LS} (\alpha \overline{LS} + R) \cos \omega.
		\end{equation}
		Solving \cref{eq:invar1} for $\cos \omega$ the following is obtained:
		\begin{equation}
			\cos \omega = \tfrac{\overline{LS} (1 + \alpha) + 2 \alpha R}{2 (\alpha \overline{LS} + R)}
			\label{eq:omega1}
		\end{equation}
		Now, consider a future time $t\in (t_2,t_f)$. Using the law of cosines for $\triangle FHS$:
		\begin{equation}
			\rho^2 = \overline{FS}^2 + \overline{HS}^2 - 2\overline{FS}\overline{HS} \cos \omega.
			\label{eq:rho1}
		\end{equation}
		Recognizing that $\overline{HS} = \alpha \overline{FS}$ and $\overline{HS} = \overline{HQ} + R$. 
		\begin{equation}
			\rho^2 = \overline{FS}^2 + (\alpha \overline{FS} + R)^2 - 2 \overline{FS}(\alpha \overline{FS} + R) \cos \omega
			\label{eq:rho2}
		\end{equation}
		Substituting \cref{eq:omega1} into \cref{eq:rho2}, the following is obtained
		\begin{equation}
			\rho^2 = R^2 + \underbrace{\tfrac{\overline{FS}R}{(\alpha \overline{LS} + R)}}_{\text{Positive}}\underbrace{(\overline{FS} - \overline{LS})}_{\text{Negative}}\underbrace{(1 - \alpha^2)}_{\text{Positive}}.
		\end{equation}
		Therefore, the values of $t \in (t_2, t_f), \rho < R$. 
	\end{proof}
	The function that describes the target distance while being observed is given in \cref{eq:LS1}. The target is moving with unity speed and therefore the observation time is 
	\begin{equation}
		\tobs = t_f-t_2 = \tfrac{2R(\alpha + \cos \lambda_{TO,2})}{1 - \alpha^2}.
		\label{eq:observationTime}
	\end{equation}
	\begin{lemma}
		The observation time monotonically increases over the interval $\lambda_{TO,2} \in [-\pi,0]$ and monotonically decreases over the interval $\lambda_{TO,2} \in [0,\pi]$ with a maximum at $\lambda_{TO,2} = 0$.
		\label{lma:Observation}
	\end{lemma}
	
	\begin{proof}
		The observation time is as given in \cref{eq:observationTime}. 
		Taking the partial derivative of the observation time, $t_\text{obs}$, with respect to $\lambda_{TO,2}$ and setting equal to zero provides candidate extremals for $\lambda_{TO,2} \in [-\pi,\pi]$.
		\begin{equation}
			\tfrac{\partial \tobs}{\partial \lambda_{TO,2}} = 0 \Rightarrow -\tfrac{2R}{1-\alpha}\sin \lambda_{TO,2} = 0
			\label{eq:partialY}
		\end{equation}
		Therefore the observation time is a maximum when $\lambda_{TO,2} = 2 \pi n, n \in \mathbb{Z}$. The only candidate in the domain of $\lambda_{TO,2} \in [-\pi,\pi]$ is   
		\begin{equation}
			\lambda_{TO,2} = 0.
			\label{eq:observationZero}
		\end{equation}
		The speed ratio $\alpha \in (0,1)$ and $R>0$. The observation time is a maximum when $\lambda_{TO,2}=0$ and monotonically decreases from $0$ over the over the domain $\lambda_{TO,2} \in [-\pi,\pi]$. Further, investigating the sign of the partial in \cref{eq:partialY}:
		\begin{align}
			- \tfrac{2R}{1-\alpha}\sin \lambda_{TO,2} &> 0 , \lambda_{TO,2} \in [-\pi,0), \label{eq:observationPositive}\\
			- \tfrac{2R}{1-\alpha}\sin \lambda_{TO,2} &< 0 , \lambda_{TO,2} \in (0,\pi].
			\label{eq:observationNegative}
		\end{align}
		From \cref{eq:observationPositive}, observation time monotonically increases over the interval $\lambda_{TO,2} \in [-\pi,0]$. From \cref{eq:observationNegative}, observation time monotonically decreases over the interval $\lambda_{TO,2} \in [0,\pi]$. Lastly, from, \cref{eq:observationZero} the maximum observation time occurs when $\lambda_{TO,2} = 0$.
	\end{proof}
	
	\begin{lemma}
		Observation time is zero for Phase-II when $\lambda_{TO,2} \in [\cos^{-1} (-\alpha), \pi] \cup [-\pi,-\cos^{-1}(-\alpha)]$.
		\label{lma:ZeroTime}
	\end{lemma}
	
	\begin{proof}
		From \cref{eq:observationTime}, the observation time is zero when $\tobs \leq 0$; namely,
		\begin{equation}
			\tfrac{2R(\alpha + \cos \lambda_{TO,2})}{1-\alpha^2} \leq 0 \Rightarrow \alpha + \cos \lambda_{TO,2} \leq 0
			\label{eq:zeroCondition}
		\end{equation}
		re-arranging \cref{eq:zeroCondition}, the conditions for which observation time is zero occurs when 
		\begin{equation}
			\lambda_{TO,2} \geq \cos^{-1} (-\alpha),
		\end{equation}
		and by symmetry,
		\begin{equation}
			\lambda_{TO,2} \leq -\cos^{-1} (-\alpha).
		\end{equation}
		Since the domain of $\lambda_{TO,2} \in [-\pi,\pi]$, the regions for which observation time is zero occurs when $\lambda_{TO,2} \in [\cos^{-1} (-\alpha), \pi] \cup [-\pi,-\cos^{-1}(-\alpha)]$.
	\end{proof}
		
	
	\begin{lemma}
		The maximum possible observation time is $\overline{t}_{\text{obs}} = \tfrac{2R}{1-\alpha}$.
		\label{lma:maxPossible}
	\end{lemma}
	\begin{proof}
		From \cref{eq:observationTime}, the observation time is 
		\begin{equation}
			\tobs = \tfrac{2R(\alpha + \cos \lambda_{TO,2})}{1 - \alpha^2}.
			\label{eq:obs2}
		\end{equation}
		From \Cref{lma:Observation}, the maximum observation time occurs when $\lambda_{TO,2} = 0$. Substitution of $\lambda_{TO,2} = 0$ into \cref{eq:observationTime} results in
		\begin{equation}
			\label{eq:maxObservation}
			\overline{t}_{\text{obs}} = \tfrac{2R(\alpha + 1)}{1 - \alpha^2} = \tfrac{2R(\alpha + 1)}{(1 - \alpha)(1 + \alpha)} = \tfrac{2R}{1-\alpha}.
		\end{equation}
	\end{proof}

	
	\section{Phase - I: Approach Phase} \label{sec:ApproachPhase}
	The objective in Phase-I is to place the observer in a favorable position,  to maximize the time of observation in Phase-II. From \Cref{lma:Observation}, more favorable positions are located lower on the observation disk for Phase-II. Much like Phase-II, because the observer is holonomic; the costates of Phase-I are constant; and, therefore, the optimal strategy for Phase-I is a straight-line trajectory. This stems from the Hamiltonian for Phase-I,
	\begin{equation}
		\mathscr{H}_{\text{I}} = p_{x_O} \alpha \cos \psi_{O,1} + p_{y_O}\alpha \sin \psi_{O,1} + p_{y_T}.
	\end{equation}
	Just as shown in \cref{eq:NOCu1,eq:constLambda,eq:optimalControl1}, the optimal heading of the observer for Phase-I is a straight line. Therefore, the selection of the optimal heading $\psi_{O,1}$ that maximizes the overall observation time in Phase-II is of interest -- namely, by \cref{eq:obs2},   
	\begin{equation}
		\begin{aligned}
			\psi_{O,1}^* \ = \ \mathop{\text{argmax}}_{\psi_{O,1}} \tobs \ = \ \mathop{\text{argmax}}_{\psi_{O,1}} \tfrac{2R \alpha + \cos \lambda_{TO,2}}{1-\alpha^2}.
		\end{aligned}
		\label{eq:phase1max}
	\end{equation}
	
	As shown in \Cref{lma:maxPossible}, the maximum observation in the second phase occurs when $\lambda_{TO,2} = 0$. Furthermore, by \Cref{lma:Observation}, the observation time monotonically increases over the interval $\lambda_{TO,2} \in [-\pi,0]$ and decreases over the interval $\lambda_{TO,2} \in [0,\pi]$. In the problem definition, $\alpha$ and $R$ are constant and do not depend upon the heading of the observer. This allows \cref{eq:phase1max} to be separated and re-written in a simpler form as follows:
	\begin{equation}
		\begin{aligned}
			\psi_{O,1}^* 
			&= \mathop{\text{argmax}}_{\psi_{O,1}} \left( \tfrac{2R \alpha}{1-\alpha} + \tfrac{1}{1-\alpha^2} \cos \lambda_{TO,2} \right)\\
			&= \mathop{\text{argmax}}_{\psi_{O,1}} \underbrace{\tfrac{2R \alpha}{1-\alpha}}_{\text{Const.}} + \mathop{\text{argmax}}_{\psi_{O,1}} \underbrace{\tfrac{1}{1-\alpha^2}}_{\text{Const.}} \cos \lambda_{TO,2} \\
			&= \mathop{\text{argmax}}_{\psi_{O,1}} \cos \lambda_{TO,2} \\
		\end{aligned}
		\label{eq:phase1max2}
	\end{equation}
	
	Over the interval $\lambda_{TO,2} \in [-\pi,\pi]$, the maximum of $\cos \lambda_{TO,2}$ occurs when $\lambda_{TO,2}$ is 0. And, due to the symmetry of the $\cos(\cdot)$ function over the domain $[-\pi,\pi]$, the absolute value of the argument should be minimized to maximize the $\cos(\cdot)$ of that argument. Because the domain of $\lambda_{TO,2} \in [-\pi,\pi]$, this fact allows the maximization problem in \cref{eq:phase1max2} to be equivalently written as a minimization problem:
	\begin{equation}
		\psi_{O,1}^* = \mathop{\text{argmax}}_{\psi_{O,1}} \cos \lambda_{TO,2} = \mathop{\text{argmin}}_{\psi_{O,1}} |\lambda_{TO,2}|.
		\label{eq:phase1max3}
	\end{equation}
	
	Furthermore, the initiated point of contact of the target by the observer at $t_2$ is a function of the state variables,
	\begin{equation}
		\lambda_{TO,2} = \tan^{-1}\begin{pmatrix} \tfrac{x_O(t_2)}{y_O(t_2) - y_T(t_2)}\end{pmatrix}
		\label{eq:anglePhase2}.
	\end{equation}
	The angle $\lambda_{TO,2}$, the location of the observer at contact $M = (x_O(t_2),y_O(t_2))$, and the target at the time of contact by the observer $L = (0,y_T(t_2))$ are illustrated in \Cref{fig:overallGeometry}. Inserting \cref{eq:anglePhase2} in \cref{eq:phase1max3}, the minimization problem as a function of the state variables becomes
	\begin{equation}
		\psi_{O,1}^* = \mathop{\text{argmin}}_{\psi_{O,1}} \left|\tan^{-1}\begin{pmatrix} \tfrac{x_O(t_2)}{y_O(t_2) - y_T(t_2)}\end{pmatrix}\right|.
		\label{eq:anglePhase3}
	\end{equation}
	The arc-tangent function is an odd function defined for all real numbers therefore \cref{eq:anglePhase3} can be re-written as
	\begin{equation}
		\psi_{O,1}^* = \mathop{\text{argmin}}_{\psi_{O,1}} \left| \tfrac{x_O(t_2)}{y_O(t_2) - y_T(t_2)} \right|.
		\label{eq:anglePhase4}
	\end{equation}
	This minimization means that the optimal strategy for Phase-I is for the observer to take a heading for which the target contacts the observation disk as low as possible, thus maximizing the total observation time in Phase-II. 
	
	
	
	
	\subsection{Decision Line}
	The state space is partitioned into regions $\mathscr{B}_1$, $\mathscr{B}_2$ and $\mathscr{B}_3$. The usage of a Decision Line for partitioning the state space is used for obtaining the optimal control of the observer.
	
	\begin{definition}
		The Decision Line ($\DL$) is the locus of points, $\DL \equiv \lbrace Z \rbrace$; where the Apollonius circle, whose foci are $Z$ and $T$ with associated speed ratio parameter, $\alpha$, is tangent to the target's path. 
		\label{def:DL}
	\end{definition}
	
	Points in the $(\hat{x}, \hat{y})$ frame that are \textit{above} the DL will have an Apollonius circle (w.r.t.\ the target, $T$, and associated speed ratio parameter, $\alpha$) that crosses the $\hat{y}$-axis, while points \textit{below} the DL will have Apollonius circles which do not intersect the $\hat{y}$-axis.
	The DL will be useful in determining which points on the observation disk (if any) can make contact with the target at the end of Phase-I (beginning of Phase-II).
	
	\begin{definition}
		The point $W$ is the point on the observation disk centered at $O$, whose radius is $R$, that is located at $(x_O, y_O-R)$. 
	\end{definition}
	The point $W$ is located at the \emph{bottom} of the observation disk and is important for determining the optimal strategy for the observer in the Phase-I, as will be seen later.
	
	\begin{lemma}
		The angle of the $\DL$ with respect to the $\hat{x}$-axis is $\theta_\DL$ = $\cos^{-1} \alpha$.
		\label{lma:DL}
	\end{lemma}
	
	\begin{proof}
		Because the target is moving vertically along the y-axis of the Cartesian fixed frame, the center of the tangent circle to the target's path is horizontal to the point of tangency. The radius of the Apollonius circle which is tangent to the target's path has radius, $R = \tfrac{ \alpha d }{1 - \alpha^2}$, where $\alpha$ is the speed ratio parameter and $d$ is the separation distance between the foci $T$ and $Z$. Furthermore, the distance between $Z$ and the center of the Apollonius circle is $\overline{ZC} = \tfrac{\alpha^2 d}{1-\alpha^2}$. These distances are labeled in \Cref{fig:DLGeom} to assist the reader in their visualization of the Apollonius circle which is tangent to the target's path.
		\begin{figure}[htpb]
			\centering
			\begin{tikzpicture}[scale=1, vect/.style={->,shorten >=0.5pt,>=latex, line width = 0.2}]

\tkzDefPoint(4,4){C} 		
\tkzDefPoint(0,0){T} 		
\tkzDefPoint(0,4){A} 		
\tkzDefPoint(2.75,2.75){Z} 		
\tkzDefPoint(1.5,0){xhat} 	
\tkzDefPoint(0,5){yhat}     
\tkzDrawPoints(C) 			
\tkzDrawPoints(T) 			
\tkzDrawPoints(Z) 			
\tkzDrawPoints(A) 			

\tkzLabelPoints[above right](C) 		
\tkzLabelPoints[above left](T) 		
\tkzLabelPoints[above left](Z) 		
\tkzLabelPoints[above right](A)		
\tkzDrawArc[line width=3,color=black, delta=15, opacity=0.5, color = green!50!black](C,A)(T)
\draw(T) node [right, scale = 1, shift={(0.2,2.8)}]{Apol$(T,Z)$}; 	
\tkzDrawSegment[dim={$~d~$,-45pt,below right},line width=0.5](T,Z)
\tkzDrawSegment[dim={$\frac{d \alpha^2}{1 - \alpha^2}$,-45pt,},line width=0.5](Z,C)
\tkzDrawSegment[dim={$\; \frac{d \alpha}{1 - \alpha^2} \;$ ,16pt,},line width=0.5](A,C)

\tkzDrawSegments[->,>=latex, line width=0.25](T,xhat)
\tkzDrawSegments[->,>=latex, line width=0.25](T,yhat)
\tkzDrawSegments[-, line width=0.25](T,C)
\tkzDrawSegments[-, line width=0.25](A,C)

\draw(xhat) node [right]{$\hat{x}$};
\draw(yhat) node [left]{$\hat{y}$};

\tkzMarkAngle[size = 1, color = black,vect, mark=0, line width=0.3](A,C,T)
\tkzLabelAngle[pos = 1.4, scale=1, shift={(0,0.1)}](A,C,T){$\theta_{DL}$}
\tkzMarkAngle[size = 0.8, color = black,vect, mark=0, line width=0.3](xhat,T,C)
\tkzLabelAngle[pos = 1.2, scale=1, shift={(0,-0.1)}](xhat,T,C){$\theta_{DL}$}
\tkzMarkRightAngle[scale = 1](T,A,C)
\tkzMarkRightAngle[scale = 1](xhat,T,A)


\end{tikzpicture}
			\caption{The Decision Line $\DL$ determines if a vehicle is capable of reaching the target vehicle's path. If a pursuing vehicle is under the decision line, then it can not reach the non-maneuvering target.}
			\label{fig:DLGeom}
		\end{figure}
		From the geometry of Apollonius Circle, Apol$(T,Z)$ and $\triangle CAT$; as shown in \Cref{fig:DLGeom}, the angle of the decision line is found from solving the following for $\theta_\DL$,
		\begin{equation}
			\cos \theta_\DL = \tfrac{\tfrac{d \alpha }{ 1 - \alpha^2 }}{d + \left( \tfrac{d \alpha^2}{1-\alpha^2} \right)} = \alpha.
			\label{eq:DLangle}
		\end{equation}
		Therefore, the angle of the decision line with respect to the target's location is 
		\begin{equation}
			\theta_\DL = \cos^{-1} \alpha.
		\end{equation}
		Furthermore, the $\DL$ is a line because $\theta_\DL$ is independent of the distance between $T$, $Z$, and $d$. Thus any point whose angle with respect to $T$ is $\theta_\DL$ is in $\DL$. Therefore, $\DL$, must be a straight line.
	\end{proof}
	
	\begin{definition}
		$\mathscr{B}_1$ is the region of the state space where no observation is possible: $\mathscr{B}_1 \triangleq \lbrace \mathbf{x} \;|\; \tobs = 0 \rbrace$.
		\label{def:B1}
	\end{definition}
	
	\begin{definition}
		$\mathscr{B}_3$ is the region of the state space where the optimal observation time is the maximum possible observation time: $\mathscr{B}_3 \triangleq \lbrace \mathbf{x} \;|\; \tobs = \tfrac{2R}{1 - \alpha} \rbrace$.
		\label{def:B3}
	\end{definition}
	
	\begin{definition}
		$\mathscr{B}_2$ is the region of the state space where the optimal observation time is bounded between zero and the maximum observation time: $\mathscr{B}_2 \triangleq \lbrace \mathbf{x} \;|\; 0 < \tobs < \tfrac{2R}{1 - \alpha} \rbrace$. By substraction from the state space $\mathbf{x} \in \mathbb{R}^3$, $\mathscr{B}_2 = (\mathscr{B}_1 \cup \mathscr{B}_3)'$.
		\label{def:B2}
	\end{definition}
	
	\begin{lemma}
		If the point at the bottom of the observation disk, $W$, is on or above the decision line, then the state $\mathbf{x} \in \mathscr{B}_3$. 
		\label{thm:B3}
	\end{lemma}
	
	\begin{proof}
		By \Cref{lma:DL}, any point in the Cartesian space above or on the decision line may reach the target. If a point on the bottom of the observation disk lies on or above the decision line then the bottom of the observation disk may reach the target, corresponding to $\lambda_{TO,2} = 0$, and the observation time is therefore $\tobs = \tfrac{2R}{1 - \alpha}$.
	\end{proof}
	
	\begin{lemma}
		If the observation disk lies beneath the decision line and intersects with the decision line at 1 or fewer points then the state $\mathbf{x} \in \mathscr{B}_1$.
		\label{thm:B1}
	\end{lemma}
	
	\begin{proof}
		By \Cref{lma:DL}, any point in the Cartesian space above or on the decision line may reach the target. If the entire observation disk lies beneath the decision line, then the observer is unable to reach the target for any amount of time. If the decision line intersects the observation at 1 point, then it does so tangentially. Because the decision line has an angle $\theta_{DL} = \cos^{-1} \alpha$, this means that the tangent point occurs at the angle defined by the limit of the observation time from \Cref{lma:ZeroTime}.
	\end{proof}
	
	\begin{lemma}
		If the decision line intersects the observer's disk at 2 points, then the state $\mathbf{x} \in \mathscr{B}_2 \cup \mathscr{B}_3$. Consequently, the optimal observation time is non-zero.
		\label{thm:nonZero}
	\end{lemma}
	
	\begin{proof}
		If a line intersects a circle at two real points, it can not be tangential to the circle, and by definition must intersect the circle~\cite[pp. 459]{rhoad1997geometry}. This requires the line to create a chord in the circle. By \Cref{lma:DL}, any point in the Cartesian space above or on the decision line may reach the target. Since two such points exist where the DL crosses the observation disk, via. a chord, it is possible for the observer to reach the target and $\tobs \neq 0$. By \Cref{def:B1}, all states where $\tobs = 0$ belong to $\mathscr{B}_1$; therefore, $\mathbf{x} \notin \mathscr{B}_1$, and therefore  by \Cref{def:B2} and \Cref{def:B3} $\mathbf{x} \in \mathscr{B}_2 \cup \mathscr{B}_3$.
	\end{proof}
	
	\begin{lemma}
		If the decision line intersects the observer's disk at 2 points and the x-coordinate of either of said points is less than the x-coordinate of $O$, then the state $\mathbf{x} \in \mathscr{B}_2$.
		\label{thm:B2}
	\end{lemma}
	
	\begin{proof}
		By \Cref{thm:nonZero}, If the decision line intersects the observer's disk at 2 points, then the state $\mathbf{x} \in \mathscr{B}_2 \cup \mathscr{B}_3$. Consequently, the optimal observation time is non-zero. Consider each intersection point: $I_A$ and $I_B$ whose coordinates are $(x_{IA}, y_{IA})$ and $(x_{IB},y_{IB})$ respectively. Recall that the observer location is $O$ and its Cartesian coordinate is $(x_O,y_O)$. Consider the following by contradiction: If both $x_{IA} \geq x_O$ and $x_{IB} \geq x_O$, then the DL creates a chord whose points are to the right of $O$. Therefore, the point at the bottom of the observation disk is on or above the DL. Therefore, from \Cref{thm:B3}, $\mathbf{x} \in \mathscr{B}_3$. 
		
		If either $x_{IA} < x_O$ or $x_{IB} < x_O$, then the state does not belong to $\mathscr{B}_1$ (because there are two points) or $\mathscr{B}_3$, for this would require both $x_{IA} \geq x_O$ and $x_{IB} \geq x_O$. Therefore, by \Cref{def:B2}, $\mathbf{x} \in \mathscr{B}_2$.
	\end{proof}
	
	\subsection{Optimal Observer Strategy - Phase I}
	The optimal heading for $O$ in the first phase is dependent upon the speed ratio parameter, $\alpha$, the observation range, $R$, and the initial state, $\mathbf{x}(t_0) = [x_O(t_0), y_O(t_0), y_T(t_0)]^\transpose$. The regions $\mathscr{B}_1$, $\mathscr{B}_2$, and $\mathscr{B}_3$, and the decision line are shown in \cref{fig:phaseDiagram}. Three targets for an observer with observation range, $R$, and speed ratio parameter, $\alpha$, are shown in order to highlight the various regions where the optimal observation time is maximum, non-zero, and, zero. 
	\begin{figure}[htpb]
		\centering
		\includegraphics[width=3.5in]{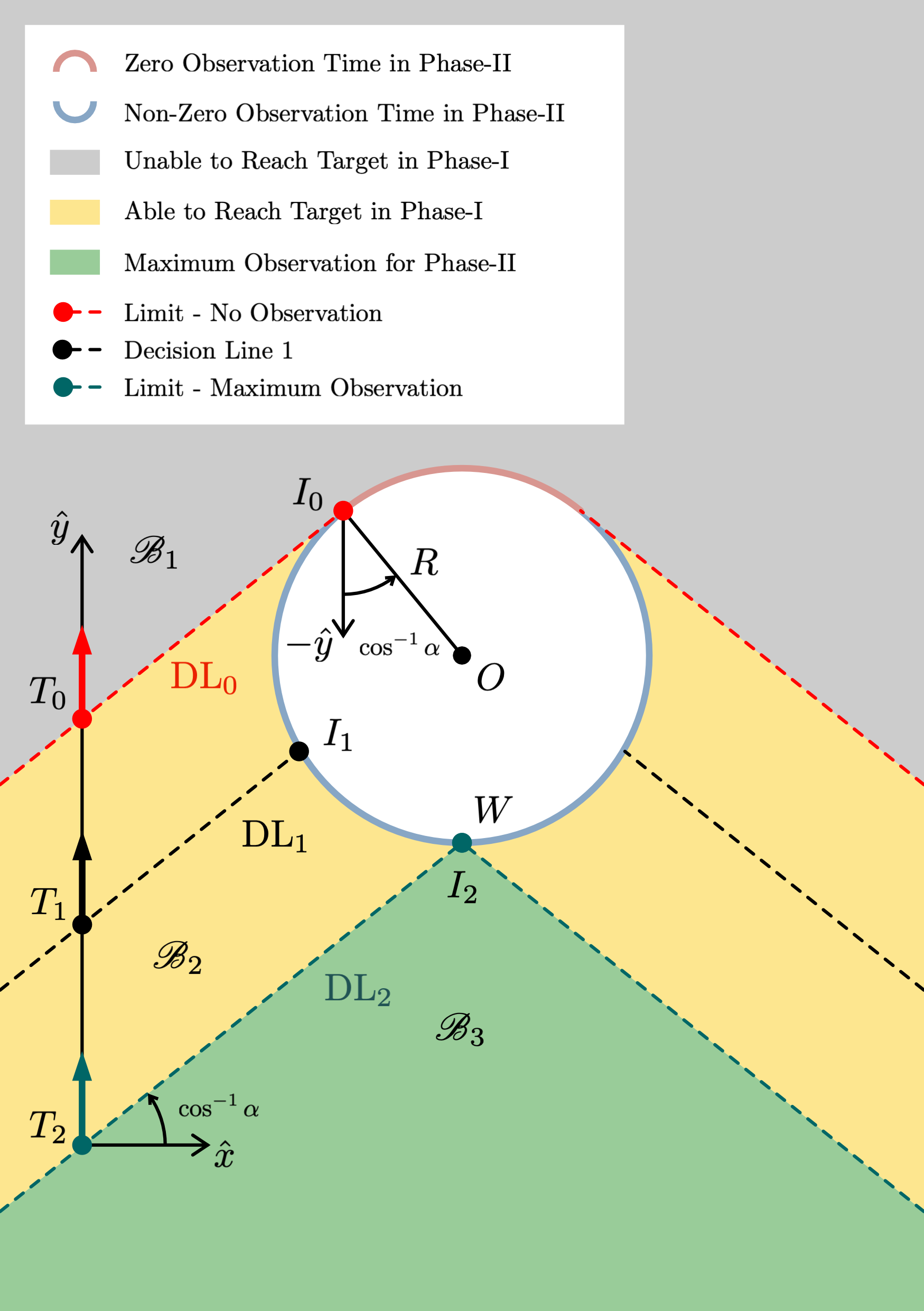}
		\caption{The outcome of the optimal observation scenario is dictated by the relative position of the target to the observer, the speed ratio parameter, $\alpha$, and the observation range, $R$. At initial time, if the target is located in the grey region then the observer is unable to reach the target. If the target is located in the yellow or green regions, relative to the observer's location $O$, then the observer can observe the target for some non-zero amount of time. }
		\label{fig:phaseDiagram}
	\end{figure}
	
	\subsubsection{Case 1}
	First, consider $T_0$. At initial time, $\mathbf{x}_0 \in \mathscr{B}_1$. By \Cref{thm:B1}, the observer can not reach the target and the $\tobs = 0$.
	Consequently, the optimal control, $\psi_{O,1}^*$, is undefined.
	
	If the target is located in the yellow or green regions, relative to the observer's location $O$, then the observer can observe the target for some non-zero amount of time.
	
	\subsubsection{Case 2}
	
	Next, consider $T_1$. At initial time, $\mathbf{x}_1 \in \mathscr{B}_2$. By \Cref{def:B2}, the observer reaches the target at some angle that is not directly in-front of the target.
	
	\begin{lemma}
		\label{lma:observation_phase1_case2}
		If $\mathbf{x} \in \mathscr{B}_2$, the optimal observation time is
		\begin{equation}
			\label{eq:time_phase1_case2}
			\tobs = \tfrac{2R}{1-\alpha^2} \left( \alpha + \tfrac{y_O - y_I}{R} \right),	
		\end{equation}
		where
		%
		\begin{equation}
			\label{eq:xIyI_phase1_case2}
			\begin{aligned}
				y_I=&\; m_{\text{DL}} x_I,\\
				x_I=&\; \alpha^2 \left( x_0 + y_O m_{\text{DL}}  - \sqrt{(x_O + y_O m_{\text{DL}})^2 - \tfrac{\sigma}{\alpha^2}} \right),\\
				m_{\text{DL}}=&\; \tfrac{\sqrt{1 - \alpha^2}}{\alpha}, \quad 
				\sigma=x_O^2+y_O^2-R^2.
			\end{aligned}
		\end{equation}
	\end{lemma}
	\begin{proof}
		If $\mathbf{x} \in \mathscr{B}_2$ then the DL must intersect the observation disk at two points (from \Cref{thm:nonZero}) and one or both of these intersections has an $x$-coordinate less than that of the observer, $O$ (from \Cref{thm:B2}).
		Let the point $I$ be defined as the intersection of the DL with the observation disk with the smaller $x$-coordinate (c.f.\ \cref{fig:overallGeometry}).
		The coordinates of $I$ can be found by computing the intersection of the DL with the observation disk.
		The equations describing the DL and observation disk can be written, respectively, as
		
		\begin{equation}
			y = \tfrac{\sqrt{1 - \alpha^2}}{\alpha} x  \label{eq:decisionLineY}
		\end{equation}
		\begin{equation}
			\left( x - x_O \right)^2 + \left( y - y_O \right) ^2 = R^2. 
			\label{eq:observationCircleE}
		\end{equation}
		
		%
		%
		Substituting \cref{eq:decisionLineY} into \cref{eq:observationCircleE} yields a quadratic equation in $x$ corresponding to the two intersections of the DL and observation disk. The point $I$ has been defined as the left-most intersection, and thus the $-$ case of the quadratic equation is taken and the result is the $x_I$ in \cref{eq:xIyI_phase1_case2}. Then $y_I$ is found by substituting $x_I$ into the DL equation, \cref{eq:decisionLineY}.
		
		The point $I$ remains fixed w.r.t. \ the observer as it moves from station 1 $(t=t_1)$ to station 2 $(t=t_2)$.
		Therefore, the point $I$ may be considered to be akin to an agent moving with speed $\alpha$ (relative to the speed of the target).
		By \Cref{def:DL}, since $I$ lies on the DL then the Apollonius circle whose foci are $I$ and $T$ and whose speed ratio is $\alpha$ is tangent to the target's path.
		Therefore, the point $I$ can ``intercept'' the target at this tangent point.
		Because $I$ remains fixed relative to $O$ the bearing angle of the observer w.r.t.\ the target at station 2, where contact is initiated, (i.e., the point $L$ in \Cref{fig:overallGeometry}) is given by 	%
		\begin{equation*}
			\lambda_{TO,2} = \tan^{-1}\begin{pmatrix} \tfrac{x_O - x_I}{y_O - y_I}\end{pmatrix}.
		\end{equation*}
		%
		Finally, the associated observation time in the subsequent phase is found by substituting $x_I$ and $y_I$ into the above equation and then into \cref{eq:observationTime}, yielding \cref{eq:time_phase1_case2}.
	\end{proof}
	
	\begin{lemma}
		\label{lma:heading_phase1_case2}
		If $\mathbf{x} \in \mathscr{B}_2$, the optimal observer heading is
		\begin{equation}
			\label{eq:heading_phase1_case2}
			\psi^*_{O,1} = \cos^{-1}\alpha + \tfrac{\pi}{2}.
		\end{equation}
	\end{lemma}
	\begin{proof}
		This proof makes extensive use of the geometry depicted in \Cref{fig:overallGeometry}.
		Recall that the point $I$ is defined as the leftmost intersection of the DL with the observation disk.
		Because the point $I$ is fixed w.r.t.\ the observer, it moves at speed $\alpha$ relative to the target.
		Also, from \Cref{def:DL}, the Apollonius circle whose foci are $I$ and $T$ must be tangent to target's path.
		Let $L$ be this tangent point.
		Then, by the definition of an Apollonius circle it must be that $\alpha \overline{TL} = \overline{IL}$.
		Now consider the triangle $\triangle TIL$; essentially, two sides and the angle $\angle LTI$ are known.
		The associated Law of Cosines is
		\begin{equation*}
			\overline{IL}^2 = \overline{TL}^2 + \left( x_I^2 + y_I^2 \right) - 2 \overline{TL} \sqrt{x_I^2 + y_I^2} \cos\left( \tfrac{\pi}{2} - \cos^{-1} \alpha \right)  .
		\end{equation*}
		Substituting in $\cos(\tfrac{\pi}{2} - \cos^{-1}\alpha) = \sqrt{1 - \alpha^2}$ along with $\overline{IL} = \alpha \overline{TL}$ yields
		\begin{equation}
			\overline{TL} = \tfrac{\sqrt{x_I^2 + y_I^2}}{\sqrt{1 - \alpha^2}},
			\label{eq:distanceApproach}
		\end{equation}
		and thus all 3 sides of the triangle $\triangle TIL$ are known.
		Now the Law of Cosines may be used again to determine the angle $\angle TIL$:
		\begin{equation*}
			\overline{TL}^2 = \overline{IL}^2 + \left( x_I^2 + y_I^2 \right) - 2 \overline{IL} \sqrt{x_I^2 + y_I^2} \cos \left(\angle TIL \right).
		\end{equation*}
		Substituting in  $\overline{IL} = \alpha \overline{TL}$ and the expression for $\overline{TL}$ yields $\cos\left( \angle TIL\right) = 0$ and therefore $\angle TIL = \tfrac{\pi}{2}$.
		Finally, the angle that $\overrightarrow{IL}$ makes with the positive $\hat{x}$-axis is given by \cref{eq:heading_phase1_case2}.
	\end{proof}
	
	\subsubsection{Case 3}
	
	\begin{figure}[htpb]
		\centering
		\includegraphics{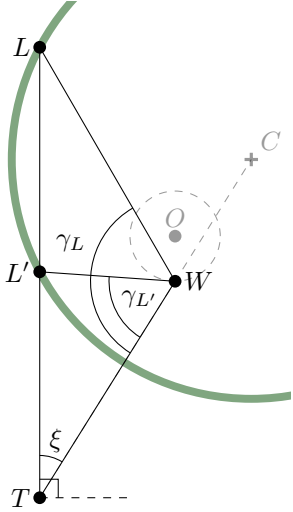}
		\vspace{-1em}
		\caption{Geometry for Case 3 wherein the lowest point on the observation disk, $W$, is above the DL resulting in two intersections of the $\hat{y}$-axis with the associated Apollonius circle.}
		\label{fig:nonuniqueHeading}
	\end{figure}
	
	Finally, consider $T_2$ in \Cref{fig:phaseDiagram}.
	The point $W$ (i.e., the lowest point of the observation disk) is on the decision line, which implies that $\mathbf{x}_2 \in \mathscr{B}_3$.
	From \Cref{def:B3} the optimal observation time must be the maximum possible observation time, $\bar{t}_{\text{obs}} = \tfrac{2R}{1 - \alpha}$.
	
	When $W$ is above the DL, then the Apollonius circle whose foci are $W$ and $T$ (with associated speed ratio $\alpha$) intersects the positive $\hat{y}$-axis twice.
	By definition, the point $W$ (moving with relative speed $\alpha$) can be moved onto any of the points along the positive $\hat{y}$-axis that are inside this Apollonius circle before the target arrives at that point.
	All of these possibilities result in $\lambda_{TO,2} = 0$ which corresponds to the maximum observation time.
	Therefore, all of these possibilities are equally optimal and thus the optimal heading for the observer is non-unique in this case.
	
	\begin{lemma}
		\label{lma:heading_phase1_case3}
		If $\mathbf{x} \in \mathscr{B}_3$, any (constant) observer heading in the range
		\begin{flalign}
			\label{eq:heading_phase1_case3}
			\psi_{O,1}^* &\in 
			\left[ \tfrac{\pi}{2}-\hspace{-0.1em}\xi+\hspace{-0.1em}\sin^{-1}\hspace{-0.3em}\left( \tfrac{\sin \xi}{\alpha} \right), 
			\tfrac{3\pi}{2} - \hspace{-0.1em}\xi - \hspace{-0.1em}\sin^{-1}\hspace{-0.3em}\left( \tfrac{\sin \xi}{\alpha} \right) \right]\hfill 
		\end{flalign}
		where
		\begin{equation}
			\label{eq:xi}
			\begin{split}
				\xi &= \sin^{-1} \begin{pmatrix} \tfrac{x_W}{\sqrt{x_W^2 + y_W^2} } 
				\end{pmatrix} = 
				\sin^{-1} \begin{pmatrix} \tfrac{x_O}{\sqrt{x_O^2 + \left( y_O - R \right)^2}} \end{pmatrix}
			\end{split}
		\end{equation}
		is optimal.
	\end{lemma}
	\begin{proof}
		Consider the general configuration for Case 3 given in \Cref{fig:nonuniqueHeading}.
		It is clear that all points along the line segment $\overline{LL'}$ lie inside the associated Apollonius circle, where $L$ and $L'$ are the two intersections of the Apollonius circle with the $\hat{y}$-axis.
		Thus, it suffices to compute the headings associated with moving the point $W$ to $L'$ and to $L$, respectively, as any heading between these will reach a point on $\overline{LL'}$ thereby achieving $\lambda_{TO,2} = 0$, giving the maximum observation time.
		First consider the triangle $\triangle TL'W$.
		From the Law of Sines and the definition of the Apollonius circle, it must be that
		\begin{equation*}
			\tfrac{\sin \gamma_{L'}}{\overline{TL'}} = \tfrac{\sin \xi}{\overline{WL'}} = \tfrac{\sin \xi}{\alpha \overline{TL'}}
		\end{equation*}
		giving $\sin \gamma_{L'} = \tfrac{\sin \xi}{\alpha}$.
		A similar relationship for the triangle $\triangle TLW$ gives $\sin \gamma_L = \tfrac{\sin\xi}{\alpha}$ as well.
		Thus it must be that $\gamma_{L'} < \tfrac{\pi}{2} < \gamma_L$.
		The associated headings can be written as $\tfrac{\pi}{2} - \xi + (\pi - \gamma_L)$ and $\tfrac{\pi}{2} - \xi + (\pi - \gamma_{L'})$.
		Taking the arcsin of the $\sin\gamma$ terms, accounting for the proper quadrant, and substituting into this expression yields \cref{eq:heading_phase1_case3}.
	\end{proof}
	
	\begin{lemma}
		\label{lma:approachTime3}
		If $\mathbf{x} \in \mathscr{B}_3$ then the time it takes for the observer to approach the target (time of Phase-I) is $\tapr = \tfrac{y_O - R}{1 - \alpha \sin \psi_{O,1}}$. Where the feasible domain of the observer headings, $\psi_{O,1}$, are as defined in \cref{eq:heading_phase1_case3} in \Cref{lma:heading_phase1_case3}.
	\end{lemma}
	
	\begin{proof}
		Because the state $\mathbf{x} \in \mathscr{B}_3$, the point $W = (x_O,y_O-R)$ is above the DL by \Cref{thm:B3}; seen in \cref{fig:approachPhase}. Let the point $L''$ be where the point W on the observation disk reaches the target's path as the observer takes one of the optimal headings in the interval described by \Cref{lma:heading_phase1_case3}. Let $T$ be the target's position and its trajectory is aligned with the $\hat{y}$-axis. 
		\begin{figure}[htpb]
			\vspace{-3.3em}
			\centering
			\includegraphics{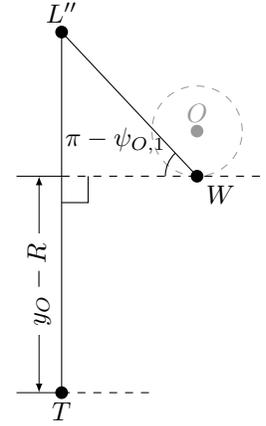}
			\vspace{-1em}
			\caption{Approach time derivation for when $\mathbf{x} \in \mathscr{B}_3$. }
			\label{fig:approachPhase}
		\end{figure}
		Consider the $\triangle TWL''$. The time that occurs until the observer reaches the target is $\tapr = \overline{TL''}/v_T$. Recall, that the speed of the target is unity by definition in \cref{eq:dynamics}. This means that the time in approach is $\tapr = \overline{TL''}$. The following can be written about $\overline{TL''}$,
		\begin{equation}
			\tapr = \overline{TL''} = (y_O-R) + \overline{WL''} \sin (\pi-\psi_{O,1}).
			\label{eq:approachTime1}
		\end{equation}
		By \Cref{lma:heading_phase1_case3} the Apollonius circle whose foci are $T$ and $W$ and speed ratio is $\alpha$ dictates the interception headings for the observer $\psi_{O,1}$ in \cref{eq:heading_phase1_case3}. Inspecting the Apollonius circle, 
		\begin{equation}
			\alpha \overline{TL''} = \overline{WL''}.
			\label{eq:speedRatio2}
		\end{equation}
		Substituting the speed ratio from \cref{eq:speedRatio2} into \cref{eq:approachTime1},
		\begin{equation}
			\overline{TL''} = (y_O-R) + \alpha \overline{TL''} \sin (\pi-\psi_{O,1}).
			\label{eq:approachTime2}
		\end{equation}
		Solving \cref{eq:approachTime2} for $\overline{TL''}$ yields,
		\begin{equation}
			\tapr = \overline{TL''} = \tfrac{y_O-R}{1-\alpha\sin \psi_{O,1}}.
			\qedhere
			\label{eq:approachTime3}
		\end{equation}
	\end{proof}
	
	\section{Complete Solution} \label{sec:CompleteSolution}
	
	Using the results of the previous Lemmas, the full solution of the Phase-I optimal control problem is summarized in the following.
	
	\begin{theorem}
		\label{thm:solution_phase1}
		The optimal heading for Phase-I that maximizes the observation time in the subsequent Phase-II, which begins once the target makes contact with the observation disk of the observer, and associated observation time is given by
		%
		\begin{equation}
			\begin{split}
				\label{eq:solution_phase1}
				&\left(\psi_{O,1}^*,\ \tobs \right) = \\
				&\quad \begin{cases}
					\left( \mathrm{undef.}, 0\right) & \mathbf{x} \in \mathscr{B}_1\\
					\left( \cos^{-1} \alpha + \tfrac{\pi}{2}, \tfrac{2R}{1-\alpha} \left( \alpha + \tfrac{y_O - y_I}{R} \right) \right)  & \mathbf{x} \in \mathscr{B}_2\\
					\left( \mathrm{Eq.} \ \cref{eq:heading_phase1_case3}, \bar{t}_{\text{obs}} \right) & \mathrm{otherwise}
				\end{cases}
			\end{split}
		\end{equation} 
		where $y_I$ is given in \cref{eq:xIyI_phase1_case2}, $\bar{t}_{\text{obs}}$ is given in \cref{eq:maxObservation}.
	\end{theorem}
	\begin{proof}
		The region in which the state lies is determined via \Cref{thm:B3,thm:B1,thm:B2,thm:nonZero}, and the associated optimal heading and observation times are given by \Cref{lma:observation_phase1_case2,lma:heading_phase1_case2,lma:heading_phase1_case3}.
	\end{proof}
	
	\begin{theorem}
		\label{thm:solution_phase2}
		The optimal heading for Phase-II that maximizes the observation time and the approach time it takes for the observer to reach the target in Phase-I is 
		\begin{equation}
			\label{eq:solution_phase2b}
			\left(\psi_{O,2}^*, \tapr \right) = \begin{cases}
				\left(\mathrm{undef.},\; \infty \right)& \text{if} \;\mathbf{x} \in \mathscr{B}_1\\
				\left(\mathrm{Eq. } \ \cref{eq:solution_phase2}, 
				\tfrac{\sqrt{x_I^2 + y_I^2}}{\sqrt{1 - \alpha^2}}  \right) & \text{if}\; \mathbf{x} \in \mathscr{B}_2\\
				\left( \tfrac{\pi}{2}, \tfrac{y_O-R}{1-\alpha\sin\psi_{O,1}} \right) & \mathrm{otherwise}
			\end{cases}
		\end{equation}
		where in \cref{eq:solution_phase2},   $\lambda_{TO,2} = \tan^{-1}\begin{pmatrix} \tfrac{x_O - x_I}{y_O - y_I}\end{pmatrix}$, and in \cref{eq:xIyI_phase1_case2}, $x_I$ and $y_I$ are provided as a function of the initial state when $\mathbf{x} \in \mathscr{B}_2$.
	\end{theorem}
	
	\begin{proof}
		First, consider the optimal headings for Phase-II. From \Cref{thm:B1}, observation time is zero therefore the associated optimal heading for the observer is undefined for Phase-II.
		By \Cref{lma:OptimalHeadingII}, the optimal heading is as described in \cref{eq:solution_phase2}, this applies when $\mathbf{x} \in \mathscr{B_2} \cup \mathscr{B}_3$. However, when $\mathbf{x} \in \mathscr{B}_3$, the observer reaches the target so that $\lambda_{TO,2} = 0$ by \Cref{lma:maxPossible} and therefore the optimal heading of the observer is $\psi_{O,2}^* = \pi/2$.

		Next, consider the approach time for Phase-I.
		By \Cref{thm:B1}, when $\mathbf{x} \in \mathscr{B}_1$, the observation time is zero, and therefore the approach phase never terminates; this is because the observer can not reach the target. 
		Next, by \Cref{lma:heading_phase1_case2}, when the distance taken by the target before being contacted by the observer is shown in \Cref{eq:distanceApproach}. Dividing this by the speed of the target (unity speed) provides the approach time, $\tapr$, when $\mathbf{x} \in \mathscr{B}_2$.
		Lastly, by \Cref{lma:approachTime3}, the approach time is described by \cref{eq:approachTime3}, when $\mathbf{x} \in \mathscr{B}_3$.
		
	\end{proof}


	\section{Scenarios} \label{sec:Examples}
	From \Cref{eq:solution_phase2,eq:solution_phase1} there exist optimal strategies for an observer to maximize the amount of time that a faster, non-maneuvering target remain inside its observation disk. To communicate the presented optimal strategies for this two-phase problem, three example scenarios are presented, highlighting interesting aspects surrounding this problem.
	
	The initial conditions of Scenario A are such that $\mathbf{x} \in \mathscr{B}_1$, the initial conditions of Scenario B are such that $\mathbf{x} \in \mathscr{B}_2$, and the initial conditions of Scenario C are such that $\mathbf{x} \in \mathscr{B}_3$.
	
	\begin{table}[H]
		\centering
		\caption{Initial conditions for each of the three examples}
		\begin{tabular}{c c c c c c c c}
			Scenario & $x_O$ & $y_O$ & $y_T$ & $\alpha$ & $R$ &  $\mathbf{x}_0$   \\ \hline \hline
			A & 8.00 & 4.00 & 0.00 & 0.60 & 2.00 & $\mathbf{x}_0 \in \mathscr{B}_1$\\
			B & 5.00 & 2.00 & 0.00 & 0.80 & 2.00 & $\mathbf{x}_0 \in \mathscr{B}_2$\\
			C & 3.00 & 6.00 & 0.00 & 0.70 & 2.00 & $\mathbf{x}_0 \in \mathscr{B}_3$
		\end{tabular}
		\label{tab:Examples}
	\end{table}
	
	\subsection{Scenario A - No Observation}
	\label{ex:A}
	As described by \cref{eq:solution_phase1}, if $\mathbf{x} \in \mathscr{B}_1$ then the optimal strategy is undefined and the observation time is zero. Although, this scenario is not as interesting as the others; it highlights when observation is not possible. 
	
	Consider Scenario A from \Cref{tab:Examples}, the initial state of the system is $\mathbf{x}_0 = (8,4,0)$, the speed of the observer with respect to the target is $\alpha = 0.60$, and the observation range is $R = 2.00$.
	
	The first step is to determine if the DL intersects the observation disk. By substituting \cref{eq:decisionLineY} into \cref{eq:observationCircleE} and solving for $x$, the intersections are obtained from the resulting quadratic equation. The DL as shown in \cref{eq:decisionLineY} for the initial conditions provided is: $y = 1.3333x$. The observation disk (centered at the location of the observer) as provided in \cref{eq:observationCircleE} is $(x-8)^2+(y-4)^2=4$. The observation disk and the DL are shown in \Cref{fig:exampleA}. Solving the two equations and two unknowns results in the following imaginary coordinates: $(x,y) = (4.8000 \pm 2.0785 i,6.4000 \pm 2.7713i)$. In this example, the roots of the quadratic equation are complex and therefore; the DL does not intersect the observation disk.
	
	Evaluating the equation for the DL at the $x$-coordinate of the observer: $y = 1.3333(8) = 10.6667$. This location is greater than the $y$-coordinate of the observer: $4.000$. Therefore the observation disk lies below the DL and $\mathbf{x}_0 \in \mathscr{B}_1$. Therefore by \Cref{thm:B1} and \Cref{thm:solution_phase1}, observation is not possible no matter the strategy of the observer.
	
	\begin{figure}[htbp]
		\centering
		\includegraphics[width = 0.8\linewidth]{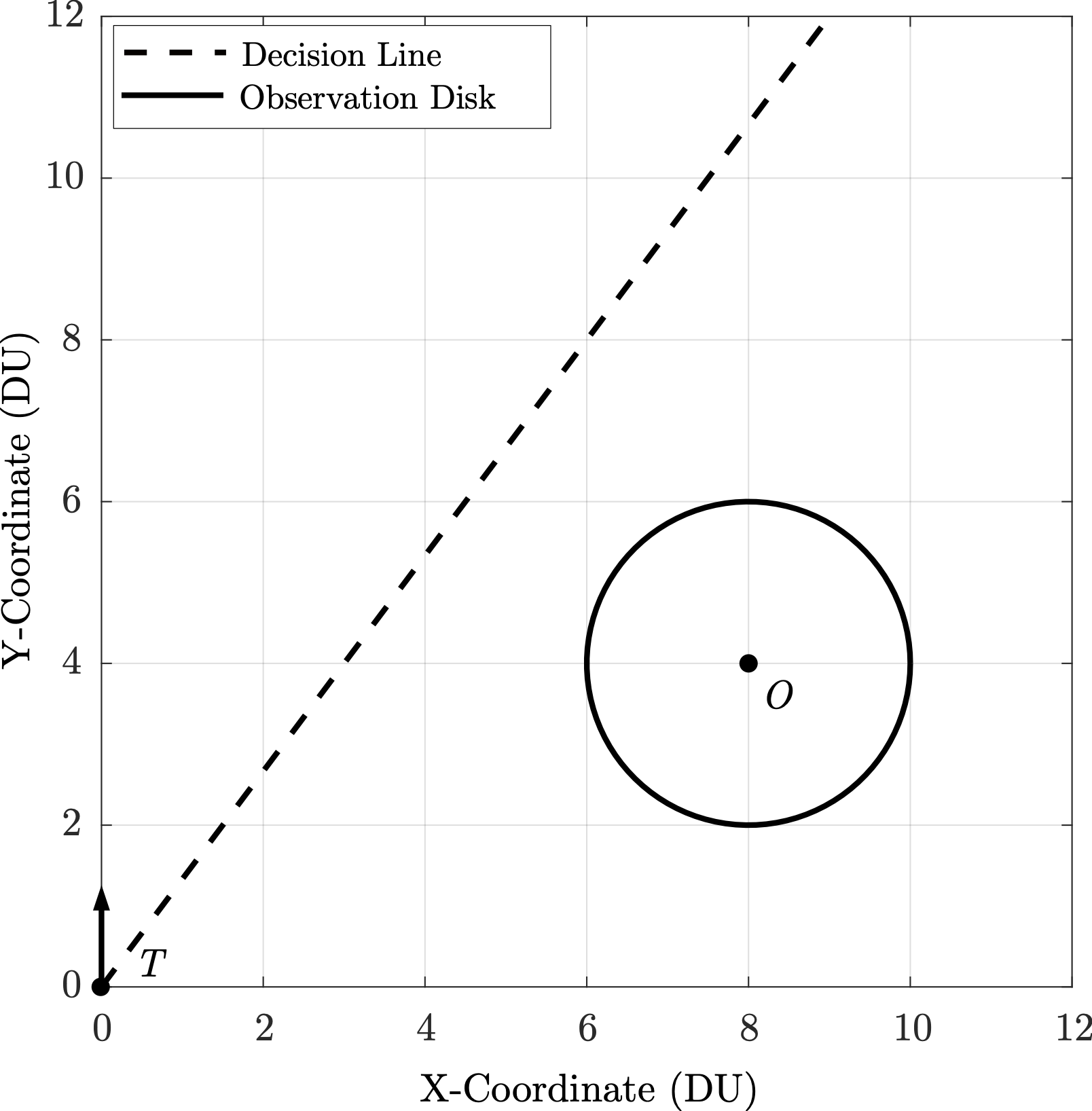}
		\caption{Example A, the observer and its observation disk lie below the decision line and $\mathbf{x}_0 \in \mathscr{B}_1$ and therefore observation is not possible.}
		\label{fig:exampleA}
	\end{figure}
	
	\subsection{Scenario B - Limited Observation}
	\label{ex:B}
	As described by \cref{eq:solution_phase1} in \Cref{thm:solution_phase1}, if $\mathbf{x} \in \mathscr{B}_2$ then there exists a unique optimal heading for the observer in Phase-I and another in Phase-II. These are used to provide the maximum possible observation of the faster non-maneuvering target.
	
	First, the intersections of the DL with the observation disk are found. In this example the DL has equation $y = 0.75x$ and the observation disk is $(x-5)^2+(y-2)^2=4$. The DL and observation disk for this example are plotted in \Cref{fig:exampleB}. By solving the quadratic equation from substituting the DL equation into the observation disk the two intersections are: $(x,y) = (3.01737,2.26303)$ and $(5.30263,3.97697)$. By \Cref{thm:B2}, since two intersections exist, $\mathbf{x} \in \mathscr{B}_2$. Once the membership of $\mathbf{x}$ is obtained, the optimal strategy for Phase-I is provided by \cref{eq:solution_phase1} in \Cref{thm:solution_phase1} and the optimal strategy for Phase-II is provided by \cref{eq:solution_phase2b} in \Cref{thm:solution_phase2}.
	\begin{figure}[htbp]
		\centering
		\includegraphics[width=\linewidth]{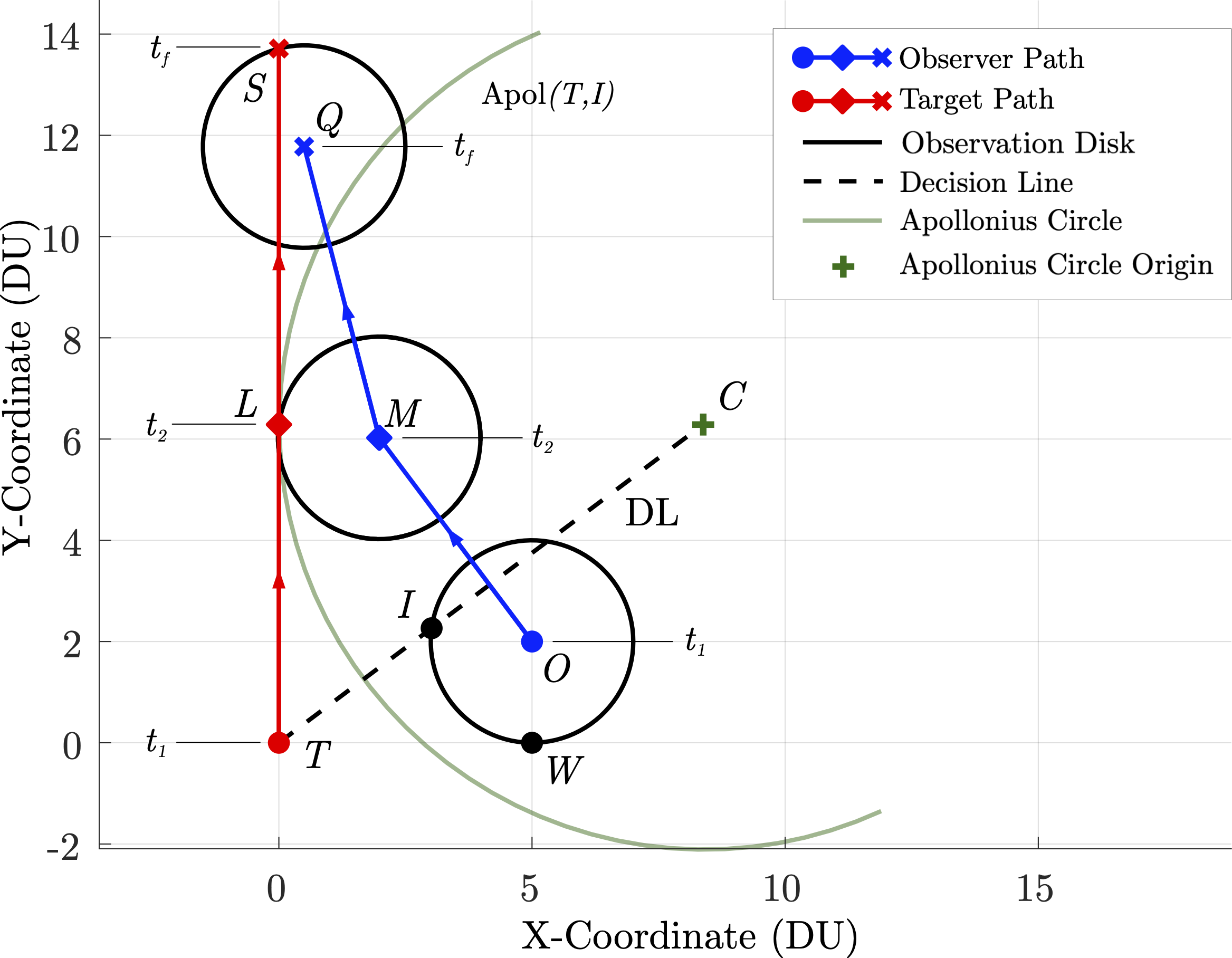}
		\caption{Example B, the observer takes a trajectory in Phase-I to maximize the observation time in Phase-II. Because the DL intersects the observation disk in two places, $\mathbf{x} \in \mathscr{B}_2$ and therefore observation time in limited.}
		\label{fig:exampleB}
	\end{figure}
	Making the appropriate substitutions, the optimal strategy is
\begin{equation*}
	\psi_{O}^*(t) = \begin{cases}
		131.4096^\circ & t \in [0,6.2862) \\
		104.4560^\circ & t \in [6.2862,13.7138]
	\end{cases}
\end{equation*}
The states at the critical times are as shown in \Cref{tab:exampleB}.
\begin{table}[htbp]
	\centering
	\caption{Scenario B - Limited Observation - Results}
	$\psi_{O,1}^*[0,6.2862) = 131.4096^\circ$, $\psi_{O,2}^*[6.2862,13.7138] = 104.4560^\circ$
	\begin{tabular}{c c c c c c}
		\hline
		$t_i$  & $t [TU]$ & $x_O [DU]$ & $y_O [DU]$ & $y_T [DU]$  \\ \hline 
		$t_1$ & \texttt{0.0000} & \texttt{5.0000} & \texttt{2.0000} & \texttt{0.0000}\\
		$t_2$ & \texttt{6.2862} & \texttt{1.9826} & \texttt{6.0232} & \texttt{6.2862}\\
		$t_f$ & \texttt{13.714} & \texttt{0.4993} & \texttt{11.777} & \texttt{13.714}
	\end{tabular}
	\label{tab:exampleB}
\end{table}
The approach time is $\tapr = 6.2862\;\text{TU}$ and the observation time is $\tobs = 7.4276\;\text{TU}$.

\subsection{Scenario C - Maximum Observation}
\label{ex:C}
As described by \cref{eq:solution_phase1} in \Cref{thm:solution_phase1}, if $\mathbf{x} \in \mathscr{B}_3$ there there exists a range of optimal headings that the observer can take in Phase-I that ensure a maximum possible observation in Phase-II. Defined by $\lambda_{TO,2} = 0$ in \cref{eq:maxObservation}, the resulting strategy for the observer in Phase-I is \cref{eq:solution_phase2b} in \Cref{thm:solution_phase2} and is $\psi_{O,2} = 90^\circ$. In this example, the limiting cases for $\psi_{O,1}^*$ are considered, highlighting the difference in outcome from implementing either limiting strategy. Both cases are plotted in \Cref{fig:exampleC}.

In this example, the DL has equation $y = 1.0202x$ and the observation disk has equation $(x-3)^2+(y-6)^2 = 4$. The intersection of the DL and observation disk have unreal solutions: $(x,y) = (4.4694 \pm 0.338326i,4.5597 \pm 0.34516i)$. Because the point $W$, whose location is $W = (x_O,y_O-R) = (3,4)$ is above the DL, we know that no intersections exist and that $\mathbf{x} \in \mathscr{B}_3$.
\begin{figure}[htbp]
	\centering
	\includegraphics[width=\linewidth]{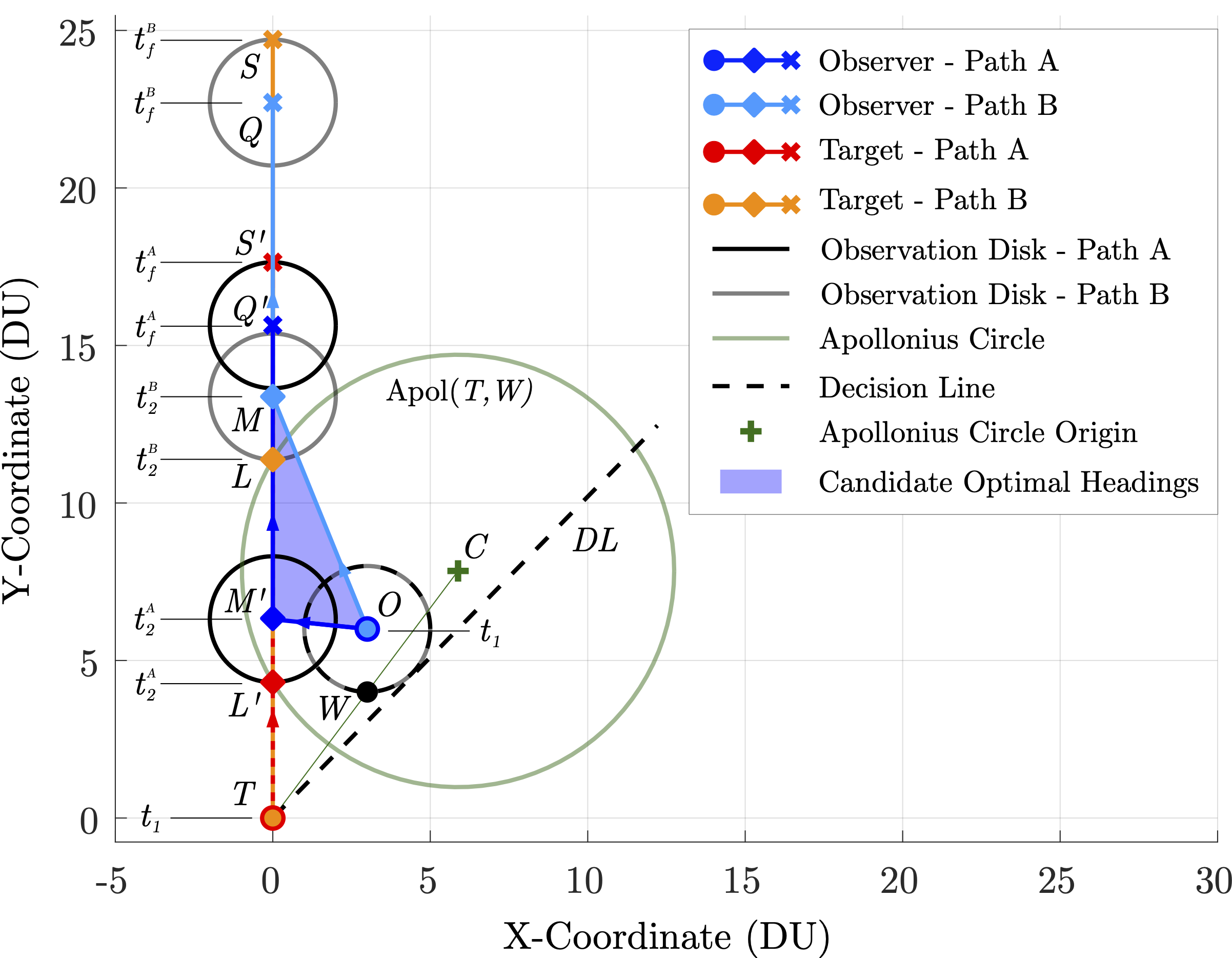}
	\caption{In the event that the point $W$ is above the DL, there exist an interval of feasible optimal observer headings for which $\tobs = \overline{t}_{\text{obs}}$. This figure shows the two limiting cases and the interval of optimal observer headings, Path A and B.}
	\label{fig:exampleC}
\end{figure}
From \Cref{thm:solution_phase1} and \Cref{thm:solution_phase2} the limiting cases for optimal observer strategies are obtained.
\begin{equation*}
	\begin{aligned}
		\psi_{O,1}^* = [112.127^\circ, 174.133^\circ], \;
		\psi_{O,2}^* = 90^\circ
	\end{aligned}
\end{equation*}
As expected, $\tapr$ varies, depending upon the heading taken for Phase-I. For the two limiting cases, $\tapr = [4.3083,11.378]$ TU. However, since $\mathbf{x}\in\mathscr{B}_3$ the maximum observation time is possible for the optimal range of observer headings; $\overline{t}_{\text{obs}} = 13.333$ TU. 

\begin{table}[htbp]
	\centering
	\caption{Scenario C - Maximum Observation - Results}
	Case A: $\psi_{O,1}^* = 174.133^\circ$\\ 
	\begin{tabular}{c c c c c c}
		\hline
		$t_i^\text{case}$  & $t [TU]$ & $x_O [DU]$ & $y_O [DU]$ & $y_T [DU]$  \\ \hline 
		$t_1^A$ & \texttt{0.0000} & \texttt{3.0000} & \texttt{6.0000} & \texttt{0.0000}\\
		$t_2^A$ & \texttt{4.3083} & \texttt{0.0000} & \texttt{6.3083} & \texttt{4.3083}\\
		$t_f^A$ & \texttt{17.642} & \texttt{0.0000} & \texttt{15.642} & \texttt{17.642}
	\end{tabular}\\ \vspace{0.5em}
	Case B: $\psi_{O,1}^* = 112.127^\circ$\\
	\begin{tabular}{c c c c c c}
		\hline
		$t_i^\text{case}$  & $t [TU]$ & $x_O [DU]$ & $y_O [DU]$ & $y_T [DU]$  \\ \hline 
		$t_1^B$ & \texttt{0.0000} & \texttt{3.0000} & \texttt{6.0000} & \texttt{0.0000}\\
		$t_2^B$ & \texttt{11.378} & \texttt{0.0000} & \texttt{13.378} & \texttt{11.378}\\
		$t_f^B$ & \texttt{24.711} & \texttt{0.0000} & \texttt{22.711} & \texttt{24.711}
	\end{tabular}
	\label{tab:exampleC}
\end{table}

\section{Conclusion} \label{sec:Conclusions}
The optimal control laws for an observer to keep a non-maneuvering constant speed target within an observation range for as long as possible have been obtained. The presented analysis and results show that the state space may be partitioned into three regions of space: no observation, limited observation, and maximum observation. Depending upon the initial conditions and problem parameters (speed ratio, $\alpha$, and observation range, $R$), this partitioning is obtained in closed form making it suitable for feedback strategies to be implemented. This is enabled by the construction of a decision line that is determined by the speed ratio parameter, $\alpha$. In order to highlight the three separate regions, three scenarios are shown, demonstrating the solutions to this optimal control problem. Future extensions of this work include observation in 3-D, observation of a maneuvering target via. a differential game formulation, and the inclusion of more observer agents. 

\section*{Acknowledgment}
\begin{small}
	This paper is based on work performed at the Air Force Research Laboratory (AFRL) \textit{Control Science Center}.
	Distribution Unlimited.
	AFRL/RQ 22-OPSEC-PR-354.
\end{small}

\bibliographystyle{IEEEtran}
\bibliography{IEEEabrv,shorttitles,bib.bib}

\begin{IEEEbiography}[{\includegraphics[width=1in, height=1.25in, clip,, keepaspectratio]{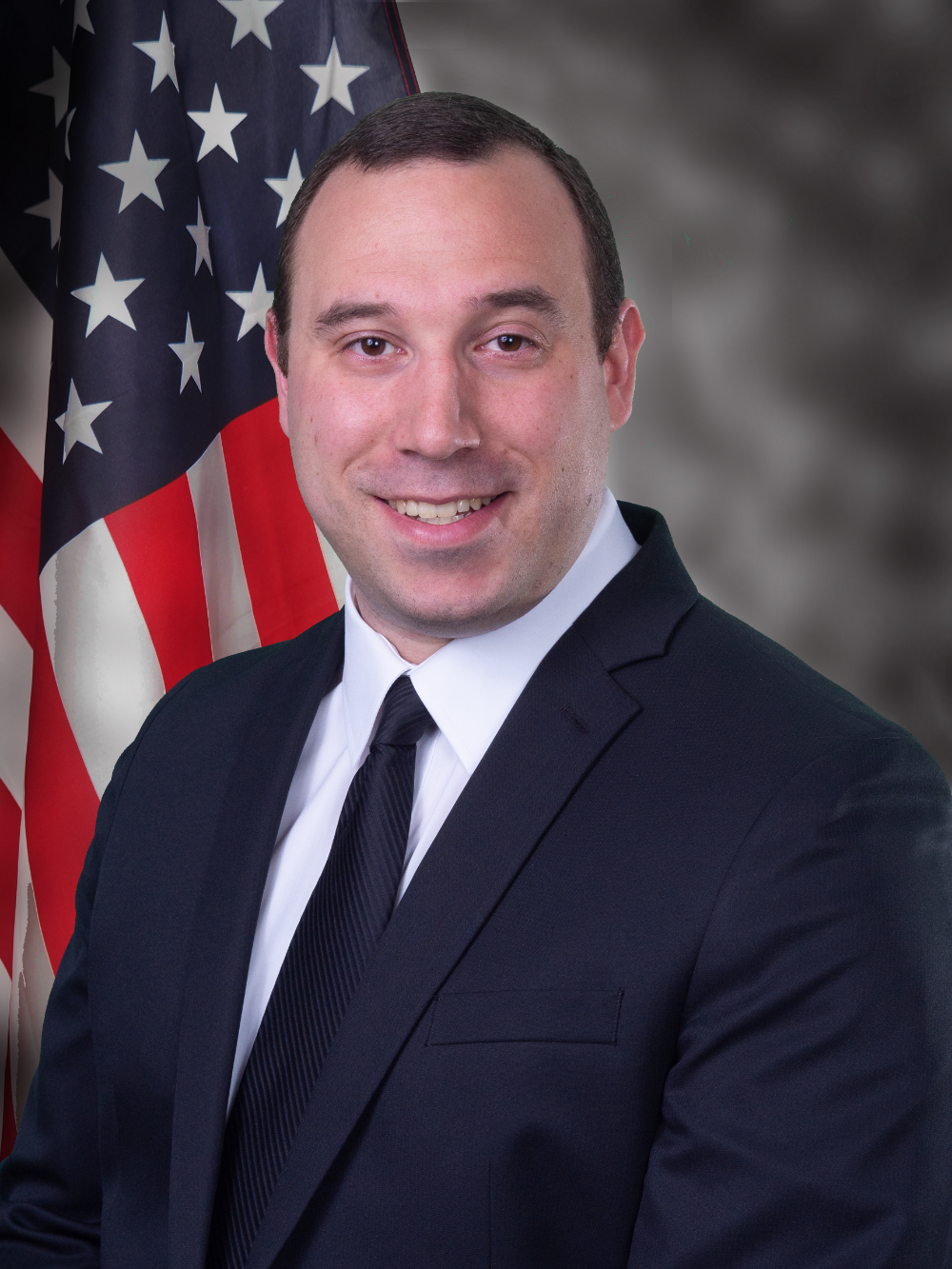}}]{Isaac E. Weintraub} Dr. Weintraub
	(S'09-M'15-SM'21) holds a Ph.D. from The Air Force Institute of Technology (2021), an M.S. in Electrical Engineering from University of Texas at Arlington (2011), and a B.S. in Mechanical Engineering from Rose-Hulman Institute of Technology (2009). He is currently an Electronics Engineer with the Control Science Center, Air Force Research Laboratory, Wright-Patterson Air Force Base, Dayton, OH, USA. 
\end{IEEEbiography}
\vspace{-4.0em}
%
\begin{IEEEbiography}[{\includegraphics[width=1in, height=1.25in, clip, keepaspectratio]{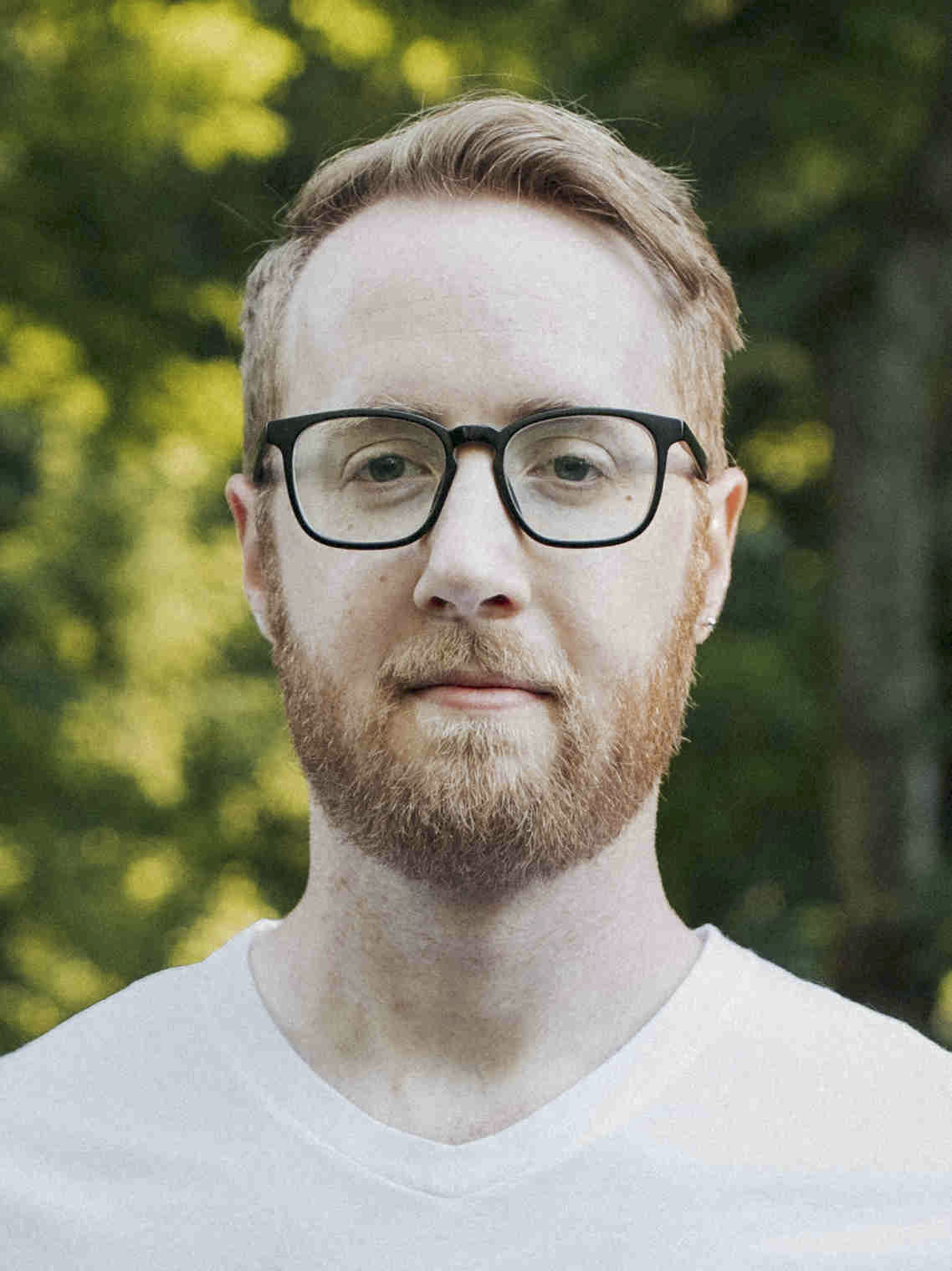}}]{Alexander Von Moll}
	Dr.\ Von Moll holds a B.S.\ in Aerospace Engineering from Ohio State (2012), an M.S.\ in Aerospace Engineering from Georgia Institute of Technology (2016), and a Ph.D.\ from University of Cincinnati (2022).
	Alex was a Department of Defense SMART Scholar, awarded in 2011 and again in 2014.
	His research interests include multi-agent systems, cooperative control, and differential games.
\end{IEEEbiography}
\vspace{-4.0em}
%
%
\begin{IEEEbiography}[{\includegraphics[width=1in, height=1.25in, clip, keepaspectratio]{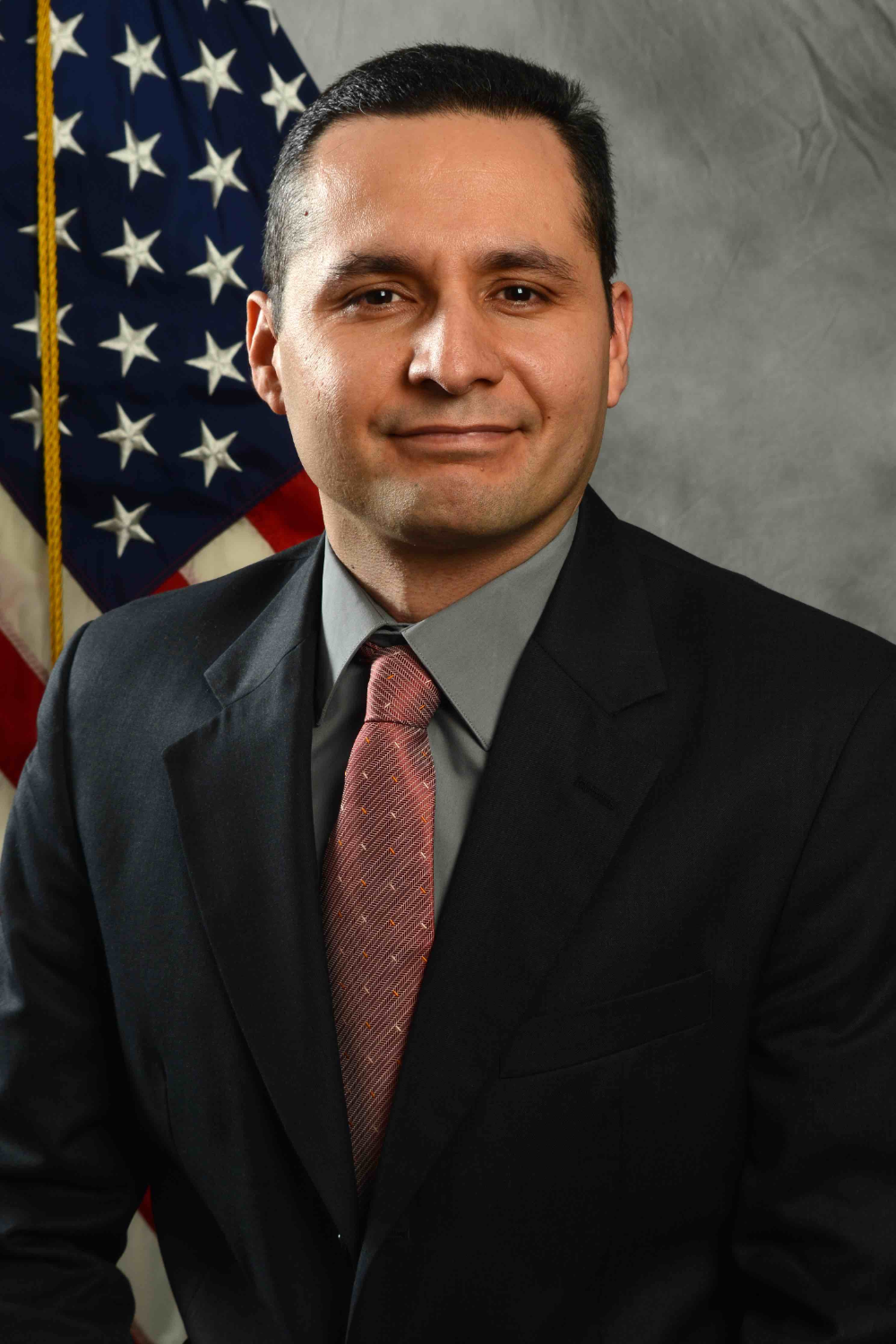}}]{Eloy Garcia}
	Dr. Garcia (SM’17) received the Ph.D. degree from the Electrical Engineering Department, University of Notre Dame, Notre Dame, IN, USA, in 2012. He also holds the M.S. degrees from the University of Illinois, Chicago, IL, USA, and the University of Notre Dame, both in electrical engineering. 
	He is currently a Research Engineer with the Control Science Center, Air Force Research Laboratory, Wright-Patterson Air Force Base, Dayton, OH, USA. 
\end{IEEEbiography}
\vspace{-4.0em}
\begin{IEEEbiography}[{\includegraphics[width=1in, height=1.25in, clip, keepaspectratio]{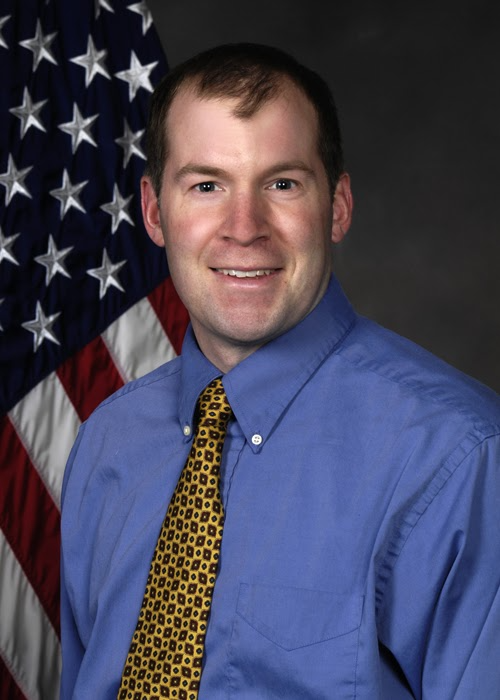}}]{David Casbeer}
	Dr. Casbeer is the Technical Area Lead over Cooperative \& Intelligent UAV Control with the Control Science Center, Aerospace Systems Directorate, Air Force Research Laboratory, where he carries out and leads basic research involving the control of autonomous UAVs with a particular emphasis on high-level decision making and planning under uncertainty.
	He received B.S. and Ph.D. degrees in Electrical Engineering from Brigham Young University in 2003 and 2009, respectively.
\end{IEEEbiography}
\vspace{-4.0em}
\begin{IEEEbiography}[{\includegraphics[width=1in, height=1.25in, clip, keepaspectratio]{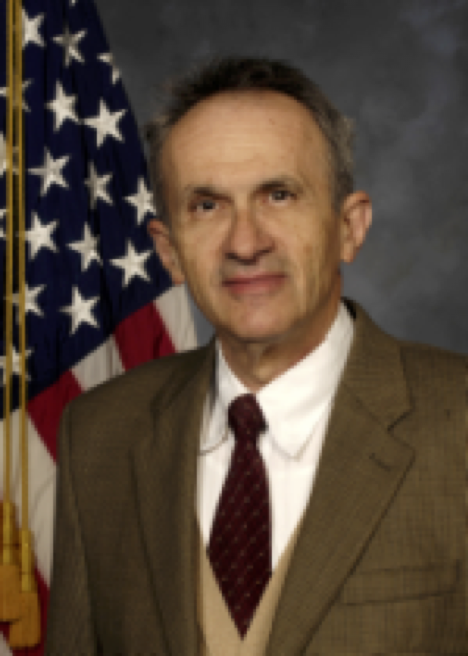}}]{Meir Pachter}
	Dr. Pachter is a Distinguished Professor of Electrical Engineering at the Air Force Institute of Technology, Wright-Patterson AFB.
	Dr. Pachter received the BS and MS degrees in Aerospace Engineering in 1967 and 1969, respectively, and the Ph.D. degree in Applied Mathematics in 1975, all from the Israel Institute of Technology.
	Dr. Pachter is interested in the application of mathematics to the solution of engineering and scientific problems.
	Dr. Pachter is a Fellow of the IEEE.    
\end{IEEEbiography}

\end{document}